\documentclass[12pt]{amsart}
\usepackage{geometry} % see geometry.pdf on how to lay out the page. There's lots.
\usepackage{pgfplots}
\usepackage{psfrag}
\usepackage{amsmath}
\usepackage{calc}
\usepackage{systeme}
\usepackage{amsfonts}
\usepackage{amsthm}
\usepackage{algorithm2e}
\usepackage[applemac]{inputenc}
\usepackage{bbold}
\usepackage[numbers]{natbib}
\newtheorem{hyp}{Hypothesis}

\author{%
  Charles Bertucci $^1$  }

\newtheorem{Theorem}{Theorem}[section]
\newtheorem{Lemma}{Lemma}[section]
\newtheorem{Rem}{Remark}[section]
\newtheorem{Def}{Definition}[section]
\newtheorem{Prop}{Proposition}[section]

\usepackage[foot]{amsaddr}

\title{Monotone solutions for mean field games master equations : finite state space and optimal stopping}
\thanks{$^1$ : CMAP, Ecole Polytechnique, UMR 7641, 91120 Palaiseau, France
}
\date{} % delete this line to display the current date

\begin{document}
\maketitle
\begin{abstract}
We present a new notion of solution for mean field games master equations. This notion allows us to work with solutions which are merely continuous. We prove first results of uniqueness and stability for such solutions. It turns out that this notion is helpful to characterize the value function of mean field games of optimal stopping or impulse control and this is the topic of the second half of this paper. The notion of solution we introduce is only useful in the monotone case. We focus in this paper in the finite state space case.\end{abstract}
\tableofcontents
\section*{Introduction}
In some sense, this paper is the fourth of a series devoted to the systematic study of mean field games (MFG for short) of optimal stopping or impulse control. Here we study the master equation associated with optimal stopping or impulse control in finite state space. In order to do so, we introduce a new notion of solution for the master equation which is of interest outside the cases of optimal stopping or impulse controls. In \citep{bertucci2018optimal}, by considering Nash equilibria of such games without common noise, we showed that those equilibria are in general in mixed strategies and translated this statement in terms of the system of partial differential equations (PDE) characterizing them. In \citep{bertucci2020fokker}, we extended this notion to the case of impulse control. In \citep{bertucci2020remark} we presented numerical methods for such problems.

\subsection*{General introduction}
The MFG theory is concerned with the study of games involving an infinite number of non-atomic players interacting through mean field terms. If such games have been studied in Economics for quite some time, a general mathematical framework has only been developed around fifteen years ago by J.-M. Lasry and P.-L. Lions. It is presented in \citep{lasry2007mean,lions2007cours}. This theory has known too many developments for us to present them all but we are going to indicate some of them. For the moment, uniqueness of Nash equilibria has been proven almost exclusively in two cases, either under a smallness condition (on the coupling between the players or on the duration of the game) or in the so-called monotone regime, see \citep{lasry2007mean,lions2007cours,porretta2015weak}. In this monotone regime, the study of the Nash equilibria of the game reduces to the study of the master equation, a PDE which is satisfied by the value function of a player seen as depending on its own state and on the measure describing the states of the other players. As soon as the state space is infinite, the master equation is an infinite dimensional PDE \citep{cardaliaguet2019master}, whereas in the finite state space case, this PDE reduces to a system of finite dimensional PDE \citep{lions2007cours}. In the case in which the game is deterministic or when the randomness is distributed in an i.i.d. fashion among the players, Nash equilibria can be characterized with a system of finite dimensional forward and backward equations \citep{lasry2007mean,cardaliaguet2019master}. Several other aspects of MFG have been studied, such as their long time average \citep{cardaliaguet2019long,cardaliaguet2012long} and the convergence of the $N$ player game toward the MFG \citep{cardaliaguet2019master,lacker2018convergence}. Let us also recall that a probabilistic approach of MFG have been developed, we refer to \citep{carmona2018probabilistic,lacker2016general} for more details on this approach. Finally let us mention that numerical methods have been developed to compute equilibria of MFG, mostly in the forward-backward setting, more details can be found in \citep{achdou2010mean,achdou2020mean,briceno2019implementation}.
In recent years, the study of MFG of optimal stopping or other "singular" controls have been the subject of a growing number of researches, namely because such games have natural applications in Economy. Concerning the case of optimal stopping, let us mention \citep{bertucci2018optimal, carmona2016mean,nutz2016mean,nutz2020convergence}, for impulse control we refer to \citep{bertucci2020fokker} and to \citep{gomes2016weakly} for optimal switching. Let us also mention that the approach of \citep{bertucci2018optimal,bertucci2020fokker} has been used by P.-L. Lions in \citep{lions2018cours} to study a case of MFG of singular controls.

\subsection*{Regularity of the solution of the master equation}
In the monotone regime, it is known since the work of J.-M. Lasry and P.-L. Lions that we can define a value function for a MFG. This is a consequence of a uniqueness property of Nash equilibria of the game. Let us recall that, formally, the monotone regime is a situation in which the players have a tendency to repel each other. Precise assumptions on the monotonicity shall be made later on.

If it is smooth, this value function satisfies the master equation. However the smoothness of this value function may be difficult to prove in general \citep{cardaliaguet2019master} and for the moment no general weak notion of solution for the master equation has been established. We provide in this paper a notion of solution which demands only for the value function to be continuous. We refer to these solutions as monotone solutions since our approach relies heavily on the monotonicity of the MFG. Let us mention that in this paper we present this solutions in the finite state space case, but that this approach can be extended. In fact the case of a continuous state space shall be treated in \citep{bertucci2020forthcomming}.

\subsection*{Mean field games with optimal stopping or singular control}
An objective of this paper is to study the master equation, in finite state space, associated to optimal stopping or impulse controls. In those situations, the players can, respectively, decide to exit the game or to change instantly their state to any other state. The main difficulty arising from the study of those games is that the evolution of a population of players using such controls is not smooth in general, independently of any regularity assumptions. This fact, already presented in \citep{bertucci2018optimal,bertucci2020fokker}, is a crucial obstacle to overcome. Indeed as the evolution of the population of player is not smooth, the formulation of the associated master equation has to take into account the fact that this measure can instantly jump from one state to another. As it is often the case in the optimal control literature, we shall see that this singular behavior for the evolution of the density of players does not translate into a loss of regularity for the value function. We believe that, in some sense, this is reminiscent of the problem of Hamilton-Jacobi equations associated with singular controls as in \citep{lasry2000classe}.

\subsection*{A comment on modeling}
Let us comment on a modeling choice we make in this paper. We intend to look at the master equation set on the whole $(\mathbb{R}_+)^d$. In the continuous state space case, in the setting of \citep{cardaliaguet2019master} for instance, this would correspond to look at the master equation posed on the set of positive measures instead that on the set of probability measure. In the appendix we explain how we can pass from one setting to another in finite state space. This choice to work on $(\mathbb{R}_+)^d$ is clearly motivated in the optimal stopping case as the mass of players is not constant (it can decrease). However, we argue that this kind of approach is also meaningful in general. Indeed, in a MFG the number of player is infinite and the choice to normalize their mass to one is only an arbitrary choice. If in a lot of cases studied in the literature the mass of player is preserved, it is also natural to consider, outside of stopping or entry game, MFG in which players leave or enter the game. For instance let us refer to \citep{claisse2019mean}. Moreover, in the case in which the mass of player is preserved, solving the master equation for any initial mass of player is a suitable strategy. To conclude this word on modeling, let us mention that to have the same point of view in the probabilistic approach of MFG, in which the measure characterizing the distribution of players is often thought as the probability measure of the state of one player, one need to consider either unnormailzed measures for the state of a player or either an additional real parameter which stands for the mass of players.

\subsection*{Structure of the paper}
The rest of this paper is organized as follows. In section 1, we present some preliminary results concerning the study of master equations in the finite state space case, in particular for situations in which there is a boundary. In section 2, we present our notion of monotone solutions. In section 3, we derive a characterization for the value function in the optimal stopping case, as well as we prove existence and uniqueness for such value functions. We present in section 4 the master equation associated with an impulse control problem. Finally section 5 is concerned with a MFG model of entry and exit in a simplified economic market which applies the concepts developed in the previous sections.

\subsection*{Notations}
We shall use the following notations.
\begin{itemize}
\item We denote by $d$ an integer greater than $1$.
\item $\mathbb{O}_d = (\mathbb{R}_+)^d$.
\item The euclidean norm of $\mathbb{R}^d$ is denoted $|\cdot|$ and its associated operator norm $\|\cdot\|$. The euclidean scalar product between $x,y \in \mathbb{R}^d$ is denoted either $x \cdot y$ or $\langle x, y \rangle$.
\item An application $A: \mathcal{O} \to \mathbb{R}^d$, defined on a subset $\mathcal{O} \subset \mathbb{R}^d$ is said to be monotone if
\begin{equation*}
\langle A(x)- A(y) , x - y \rangle \geq 0, \text{ for } x,y \in \mathcal{O}.
\end{equation*}
\item We define for $R\geq 0$, $B^1_R := \{ x \in \mathbb{O}_d, x_1 + ... + x_d \leq R\}$.
\item For $A \in \mathcal{L}(\mathbb{R}^d)$, we define 
\begin{equation*}
\|A\|_{2,-} := \inf \{ \langle \xi, A\xi\rangle | |\xi|^2 \leq 1\}.
\end{equation*}
\item We denote the term by term product between matrices or vector by $*$.
\item For $x \in \mathbb{R}^d$, we note $x\leq 0$ when all the components of $x$ are negative.
\item If $f$ is a real function and $x\in \mathbb{R}^d$, then $f(x) := (f(x_1),...,f(x_d))$.
\end{itemize}

\section{Preliminary results} 
This section introduces extensions of the results of \citep{lions2007cours} to the case in which the master equation is posed on a domain of $\mathbb{R}^d$, in particular when it is $(\mathbb{R}_+)^d := \mathbb{O}_d$. We also recall a weak form of Stegall's variational lemma at the end of this section. Let us mention the work \citep{bayraktar2019finite}, which is also concerned with the study of the master equation in finite state space. In this article the authors derived and studied a particular form of the master equation arising from the presence of a so-called Wright-Fischer common noise in the MFG. This work is radically different than ours because they use the structure of the noise to establish regularity in a non-monotone setting while we make an extensive use of monotonicity to avoid having to use regularity, without using the structure of the randomness.

The main results of this section are concerned with existence and uniqueness of solutions of the master equation

\begin{equation}\label{me}
\begin{aligned}
\partial_t U + (F(x,U)\cdot \nabla_x)U + \lambda (U - T^*U(t,Tx))&= G(x,U) \text{ in } (0,\infty) \times \mathbb{O}_d,\\
U(0,x) &= U_0(x) \text{ in } \mathbb{O}_d,
\end{aligned}
\end{equation}
or its discounted stationary counterpart
\begin{equation}\label{sme}
r U + (F(x,U)\cdot \nabla_x)U + \lambda (U - T^*U(t,Tx))= G(x,U) \text{ in }  \mathbb{O}_d.
\end{equation}
where $U : (0,\infty)\times \mathbb{O}_d \to \mathbb{R}^d$ is the value function of the MFG which is characterized by $F, G : \mathbb{O}_d \times \mathbb{R}^d\to \mathbb{R}^{2d}$, $\lambda > 0$ and $T \in \mathcal{L}(\mathbb{R}^d)$. Let us note that these PDE are not written on an open set and that no boundary conditions are imposed.

The interpretation of (\ref{me}) is that it is the PDE satisfied by the value function of the MFG. The value of the game in the state $i\in \{1,...,d\}$, when the quantity of player in each state $j \in \{1,...,d\}$ is $x_j$ and the time remaining in the game is $t\geq 0$ is : $U^i(t,x)$. The term $\lambda(U - T^* U\circ T)$ stands for the modeling of common noise and we refer to \citep{bertucci2019some} for more details on this question. The function $F$ stands for the evolution of the "density" of players and $G$ for the evolution of the value of the MFG. That is, in the case $\lambda = 0$, the characteristics of (\ref{me}) are given by $(V(t), y(t))_{0\leq t \leq t_f}$ solution of 
\begin{equation}\label{char}
\begin{cases}
\dot{V}(t) = G(y(t),V(t)), 0 < t < t_f , V(0) = U_0(y(0)),\\
\dot{y}(t) = F(y(t),V(t)), 0<t<t_f, y(t_f) = y_0,
\end{cases}
\end{equation}
for any $y_0 \in \mathbb{O}_d$, $t_f \geq 0$. The same kind of interpretation holds for (\ref{sme}) and we do not detail it here. Let us only insist on the obvious fact that the characteristics associated to (\ref{sme}) are set on an infinite time scale and that despite the fact the equation (\ref{sme}) is stationary, the evolution of the density of players is not trivial.

Because no boundary conditions are imposed on the equation, we have to restrict the set of functions $F$ which leave invariant $\mathbb{O}_d$ as well as transformation $T$ which leave invariant $\mathbb{O}_d$. Namely we shall assume the following.

\begin{hyp}\label{stab}
The function $F$ is such that for any $p \in \mathbb{R}^d$, $i\in \{1,...,d\}$
\begin{equation}\label{stabF}
x^i = 0 \Rightarrow F^i(x,p) \leq 0.
\end{equation}
Moreover $T(\mathbb{O}_d) \subset \mathbb{O}_d$.
\end{hyp}
\begin{Rem}
In certain situations, we may be force to consider a coupling term $F$ such that (\ref{stabF}) does not hold. In such a situation, we may be able to obtain, from other assumptions, that 
\begin{equation}
x^i = 0 \Rightarrow F^i(x,U(x)) \leq 0,
\end{equation}
for a solution $U$ of the master equation, which is sufficient. We refer to \citep{bertucci2020mean} for an example of such a situation.
\end{Rem}
This assumption clearly enforces the fact that $\mathbb{O}_d$ is invariant for the trajectories generated by (\ref{char}). We shall also made the following assumption in the stationary case.
\begin{hyp}\label{grow}
There exists $R > 0$ such that for all $x \in \mathbb{O}_d$ with $|x|_1 \geq R$, for all $p \in \mathbb{R}^d$,
\begin{equation}
\sum_{i = 1}^d F^i(x,p) \geq 0.
\end{equation}
Moreover, $T(B^1_R) \subset B^1_R$.
\end{hyp}
This assumption states that when the mass of player is sufficiently large ($|x|_1 \geq R$), it cannot grow anymore. Thus it has the effect to bound the region of interest when starting from an initial distribution of mass. Although it seems possible to treat situations in which the mass of players can always increase, it is not the objective of this paper.

As already mentioned several times above, monotonicity plays a crucial role in the well-posedness of (\ref{me}) and (\ref{sme}). We say that we are in the monotone regime when the following assumption is satisfied. 
\begin{hyp}\label{mon}
The functions $U_0$ and $(G,F)$ are monotone.
\end{hyp}
In order to obtain existence of solutions of either (\ref{me}) or (\ref{sme}), we shall use the following stronger assumption.
\begin{hyp}\label{smooth}
The functions $U_0$, $F$ and $G$ are lipschitz continuous. Moreover, there exists $\alpha > 0$ such that 
\begin{equation}
\langle \xi, D_x U_0(x) \xi\rangle \geq \alpha |D_xU_0(x)\xi|^2, \forall x \in \mathbb{O}_d, \xi \in \mathbb{R}^d,
\end{equation}
\begin{equation}
\begin{pmatrix}
D_x G(x,p) & D_x F(x,p) \\
D_p G(x,p) & D_p F(x,p)
\end{pmatrix}\geq \alpha \begin{pmatrix}
Id & 0 \\
0 & 0
\end{pmatrix} \forall x \in \mathbb{O}_d, p \in \mathbb{R}^d,
\end{equation}
in the order of positive matrices.
\end{hyp}
When addressing the existence of solutions of equations of the type of (\ref{sme}), we shall also make the following assumption.
\begin{hyp}\label{strongstat}
The discount rate $r$ satisfies
\begin{equation}
r > \|D_xF(x,p)\|_{2,-} \forall x \in \mathbb{O}_d,p \in \mathbb{R}^d.
\end{equation}
\end{hyp}

\subsection{Uniqueness results for the master equation in finite state space}
We now present uniqueness results concerning (\ref{me}) and (\ref{sme}). Those results are direct extensions of the results established in \citep{lions2007cours} and do need any new ideas, however, as they play a crucial role in the notion of solution we introduce in the next section, we detail their proofs.
\begin{Prop}\label{uniq1}
Under hypotheses \ref{stab} and \ref{mon}, there exists at most one smooth solution of (\ref{me}). If this solution exists, it is monotone.
\end{Prop}
\begin{proof}
Let us take $U$ and $V$ two smooth solutions of (\ref{me}). We define $W$ by
\begin{equation}
W(t,x,y) = \langle U(t,x) - V(t,y) , x - y \rangle.
\end{equation}
This function satisfies
\begin{equation}
W(0,x,y) = \langle U_0(x) - U_0(y), x - y\rangle \text{ in }\mathbb{O}_d^2,
\end{equation}
\begin{equation}
\begin{aligned}
\partial_t W +& F(x,U)\cdot \nabla_x W + F(y,V)\cdot \nabla_y W + \lambda(W - W(t,Tx,Ty))\\
 &= \langle G(x,U) - G(y,V), x-y\rangle + \langle F(x,U) - F(y,V), U - V \rangle \text{ in } (0, \infty) \times \mathbb{O}_d^2.
\end{aligned}
\end{equation}
From hypothesis \ref{mon} the right hand side of the previous two equations are positive. From lemma \ref{max1} (in appendix), we deduce that $W$ is positive for all time. We can conclude that i) $U = V$ (or otherwise $W$ should change sign around some point), ii) $U$ is monotone for all time.
\end{proof}
\begin{Prop}\label{uniq2}
Under hypotheses \ref{stab} - \ref{mon}, there exists at most one smooth solution of (\ref{sme}), and if it exists it is monotone.
\end{Prop}
\begin{proof}
Let us take $U$ and $V$ two smooth solutions of (\ref{me}). We define $W$ by
\begin{equation}
W(x,y) = \langle U(x) - V(y) , x - y \rangle.
\end{equation}
This function satisfies
\begin{equation}
\begin{aligned}
r W + F(x,U)\cdot& \nabla_x W + F(y,V)\cdot \nabla_y W + \lambda(W - W(Tx,Ty))\\
 &= \langle G(x,U) - G(y,V), x-y\rangle + \langle F(x,U) - F(y,V), U - V \rangle \text{ in }\mathbb{O}_d^2.
\end{aligned}
\end{equation}
We then conclude as in the previous proof by using this time lemma \ref{max2} in appendix.
\end{proof}

\subsection{Existence results for the master equation in finite state space}
We now turn to the questions of existence of solutions of (\ref{me}) and (\ref{sme}). As the next section gives a precise definition of a solution of (\ref{me}) and (\ref{sme}), we do not focus on the sense in which solutions satisfy the PDE but rather on how we can obtain a priori estimates.

In the monotone setting (i.e. under hypothesis \ref{mon}), an a priori estimate on the spatial gradient of the solution can be proved, exactly as it is already the case in \citep{lions2007cours}.
\begin{Prop}\label{existme}
Under hypotheses \ref{stab}, \ref{mon} and \ref{smooth}, there exists a lipschitz function $U$, solution of (\ref{me}) almost everywhere, such that for any $t_f > 0$, there exists $C$ such that
\begin{equation}
\|D_x U(t,x)\| \leq C, \forall x \in \mathbb{O}_d, t \leq t_f.
\end{equation}
\end{Prop}
\begin{proof}
This proposition can be obtained as a consequence of an a priori estimate on the solution $U$ of (\ref{me}). This idea is mostly borrowed from the lecture \citep{lions2007cours} in which this technique is presented. The only difference between this lecture and our situation is that we consider a master equation on $\mathbb{O}_d$, that is why we are not going to enter in a lot of details here but only indicate the main differences with the case in $\mathbb{R}^d$. For $U$ a smooth solution of (\ref{me}), let us define $W$ and $Z$ with
\begin{equation}
W(t,x,\xi) = \langle U(t,x), \xi \rangle,
\end{equation}
\begin{equation}
Z_{\beta,\gamma}(t,x,\xi) = \langle \nabla_x W(t,x,\xi), \xi \rangle - \beta(t) \|\nabla_x W(t,x,\xi)\|^2 + \gamma(t) |\xi|^2,
\end{equation}
for some functions $\beta$ and $\gamma$ to be defined later on. Arguing as in \citep{lions2007cours}, we deduce that $Z_{\beta,\gamma}$ satisfies
\begin{equation}\label{Zedp}
\begin{aligned}
& \partial_t Z_{\beta, \gamma} + \langle F(x,\nabla_{\xi} W), \nabla_x Z_{\beta, \gamma}\rangle + \langle D_p F (x,\nabla_{\xi} W ) \nabla_{\xi}Z_{\beta, \gamma} , \nabla_{x} W\rangle - \langle D_p G(x, \nabla_{\xi}W)\nabla_{\xi}Z_{\beta, \gamma}, \xi\rangle \\
&+  \lambda(Z_{\beta, \gamma} - Z_{\beta, \gamma}(t,Tx, T\xi ) )\\
& = \langle D_xG(x, \nabla_x W) \xi, \xi\rangle - \langle D_pG(x, \nabla_{\xi}W) \nabla_x W, \xi\rangle - \langle D_x F(x, \nabla_{\xi} W) \nabla_x W, \xi\rangle \\
&+ \langle D_p F(x, \nabla_{\xi}W) \nabla_x W, \nabla_x W\rangle\\
&- 2 \beta \langle D_x G(x, \nabla_x W)\xi, \nabla_x W\rangle + 2 \beta \langle D_x F(x, \nabla_{\xi} W) \nabla_x W, \nabla_x W\rangle\\
& + \beta \lambda \bigg{(} |\nabla_x W|^2 - 2 \langle \nabla_x W(t, Tx, T \xi ), S \nabla_x W\rangle +  | \nabla_x W(t,Tx, T \xi)|^2 \bigg{)} \\
&- \frac{d}{dt} \beta |\nabla_x W|^2  + 2 \gamma(\langle D_p F (x,\nabla_{\xi} W ) \xi , \nabla_{x} W\rangle - \langle D_p G(x, \nabla_{\xi}W)\xi, \xi\rangle)\\
& + \frac{d}{dt}\gamma |\xi|^2 + \lambda \gamma (|\xi|^2 - |T\xi - e|^2).
\end{aligned}
\end{equation}
Using the monotonicity of $U_0,G$ and $F$ and the strong monotonicity of $U_0$ and $G$, we deduce that by choosing $\gamma = 0$ and 
\begin{equation}
\beta(t) = \alpha e^{-t(2\|\nabla_x F\| + 2 \|\nabla_x G\| + \lambda(\|T\| -1)_+)},
\end{equation}
the right hand side of (\ref{Zedp}) is positive and thus that by lemma \ref{max1}, $Z_{\beta,\gamma}$ is positive for all time which yields the required a priori estimate on $U$.
Concerning the question of existence of a lipschitz function satisfying (\ref{me}), the main argument (as in \citep{lions2007cours}) is to remark that the previous technique to obtain a priori estimates still works when one add particular degenerate elliptic second order terms in (\ref{me}). This is a consequence of the so-called Bernstein method. We now indicate a particular choice of such terms.

Let us take $\epsilon > 0$ and consider $\sigma : \mathbb{R} \to \mathbb{R}$ a smooth real bounded function such that $\sigma (x) = x^2$ in a neighborhood of $0$. Assume that $U$ is a smooth solution of
\begin{equation}\label{regme}
\begin{aligned}
\partial_t U^i + F(x,U) \cdot \nabla_x &U^i - \epsilon \left(\sum_{j=1}^d \sigma(x_j)\partial_{jj}U^i\right) -  \epsilon \sigma'(x_i)\partial_i U^i\\
& + \lambda (U^i - (T^*)^iU(t,Tx)) = G^i(x,U) \text{ in } \mathbb{O}_d, \forall i \in \{1,...,d\},\\
U(0,x) = U_0(x) & \text{ in } \mathbb{O}_d.
\end{aligned}
\end{equation}
Let us remark that as the terms in $\epsilon$ in (\ref{regme}) preserves the monotonicity, we are able to adapt the previous technique to (\ref{regme}) to establish a, uniform in $\epsilon \in (0,1)$, a priori estimate on the spatial gradient of a solution of (\ref{regme}). Let us now remark that once we have this a priori estimate, from classical results on degenerate elliptic equations \citep{oleinik2012second}, existence of solutions of (\ref{regme}) can be easily obtained and that passing to the limit $\epsilon \to 0$ we deduce the existence of a lipschitz function, solution almost everywhere of (\ref{me}).
\end{proof}

We now provide an existence result for (\ref{sme}). As the proof of the following statement is very similar to the one of the previous result, we do not detail it here.
\begin{Prop}\label{existstat}
Under hypotheses \ref{stab}-\ref{strongstat}, any smooth function $U$ solution of (\ref{sme}) satisfies 
\begin{equation}
\forall x \in \mathbb{O}_d, \|D_xU(x)\| \leq C,
\end{equation}
for $C > 0$ depending only on $r,G,F,\lambda$ and $T$.
\end{Prop}
\begin{Rem}\label{rembound}
Under the assumptions of the previous proposition, existence of a lipschitz solution of (\ref{sme}) is then easy to obtain once some local boundness can be established for a solution of (\ref{sme}). Such a property can usually be obtained easily on a case by case basis. We give such an example in the next section on optimal stopping.
\end{Rem}
\begin{Rem}
The main difference for the proof of this statement compare with the time dependent case is that the functions $\beta$ and $\gamma$ in the previous proof must be chosen constant.
\end{Rem}
\begin{Rem}
We firmly believe the restriction on $r$ from hypothesis \ref{strongstat} to be mainly technical and due to the rather abstract framework in which we are working. Let us for instance mention \citep{cardaliaguet2019long} in which a lipschitz estimate is proved for the solution of a first order stationary master equation for any discount rate.
\end{Rem}

\subsection{Stegall's variational principle}
We end this section on preliminary results by a recalling a quite weak version of Stegall's lemma \citep{stegall1978optimization,stegall1986optimization,fabian2007stegall} that we shall use many times in the rest of the paper. Moreover, we present a proof of this result that we believe to be new.
\begin{Lemma}\label{stegall}
Let $\phi: \Omega \to \mathbb{R}$ be a weakly sequentially lower semi continuous function from a compact set $\Omega \subset X$ of a separable Hilbert space $X$. Then there is a dense number of points $c$ in $X$ such that $x\to \phi(x) + \langle c,x\rangle$ has a strict global minimum on $\Omega$.
\end{Lemma}
\begin{proof}
This proof relies on convex analysis. Let us consider the operator $A : X \to \mathcal{P}(X)$ defined by $A(c)$ is the set of the points at which $x\to \phi(x) + \langle c,x\rangle$ reaches its minimum over $\Omega$. Clearly for all $c$, this set is non empty and well defined. Let us check that the operator $-A$ is cyclically monotone. We consider a finite sequence $c_0,c_2,...,c_n = c_0$ and for all $i$, $y_i \in A(c_i)$. For all $i$, let us remark that
\begin{equation}
\phi(y_i) + \langle c_i, y_i\rangle \leq \phi(y_{i+1}) + \langle c_i, y_{i+1} \rangle.
\end{equation}
Rearranging we get
\begin{equation}
  \langle c_i, y_i - y_{i+1}\rangle \leq \phi(y_{i+1}) - \phi(y_i).
\end{equation}
Let us now compute
\begin{equation*}
\begin{aligned}
\sum_{i =1}^n \langle c_i - c_{i-1}, y_i \rangle & = \sum_{i =1}^n \langle c_i,  y_i - y_{i+1} \rangle,\\
& \leq \sum_{i = 1}^n \phi(y_{i+1}) - \phi(y_i),\\
& \leq 0.
\end{aligned}
\end{equation*}
Therefore the operator $-A$ is cyclically monotone. Thus we deduce that $A \subset \partial \psi$ where $\psi$ is a concave function on $X$ and $\partial \psi$ is its super-differential. From a generalization of Alexandrov theorem for separable Hilbert spaces \citep{borwein1994second}, we finally deduce the required result.
\end{proof}

\section{Monotone solutions of the master equation}\label{sectionmon}
In this section we provide a notion of solution for the master equation which does not require the solution to be differentiable with respect to the space variable. We define first our solutions in the easier case of the stationary master equation and then present the time dependent case.

\subsection{The stationary case}
The main idea we exploit in this section is somehow contained in the proof of propositions \ref{uniq2}. Namely let us remark that for the proof of proposition \ref{uniq2} to hold, we only need to have information on the solution $U$ of (\ref{sme}) at points of minima of $\phi_{V,y} : x \to \langle U(x) - V, x - y\rangle$ for $V \in \mathbb{R}^d, y \in \mathbb{O}_d$. Let us now remark that if $U$ is smooth, so is $\phi_{V,y}$ and 
\begin{equation}
\nabla_x \phi_{V,y}(x) = U(x) - V + D_x U(x)\cdot(x - y).
\end{equation}
In particular, if $x_0$ is a point of minimum of $\phi_{V,y}$ in the interior of $\mathbb{O}_d$, then 
\begin{equation}\label{eq1}
 D_x U(x_0)\cdot(x_0 - y) =V - U(x_0).
\end{equation}
The right hand side of (\ref{eq1}) does not depend on derivatives of $U$. This leads us to understand how we can generalize the notion of solution of (\ref{sme}) for a function $U$ which is not differentiable. Accordingly to this heuristic, we introduce the following definition.
\begin{Def}\label{defmons}
A function $U \in \mathcal{C}(\mathbb{O}_d, \mathbb{R}^d)$ is said to be a monotone solution of (\ref{sme}) if for any $V \in \mathbb{R}^d, y \in \mathbb{O}_d$, $R>0$ sufficiently large and $x_0$ a point of strict minimum of  $\phi_{V,y} : x \to \langle U(x) - V, x - y\rangle$ in $B^1_R$, the following holds
\begin{equation}\label{smon}
\begin{aligned}
r\langle U(x_0), x_0 - y\rangle &+ \lambda \langle U(x_0) - T^*U(Tx_0),x_0 - y\rangle \geq\\
& \langle G(x_0, U(x_0)), x_0 - y\rangle + \langle F(x_0,U(x_0)), U(x_0) - V\rangle.
\end{aligned}
\end{equation}
\end{Def}
\begin{Rem}
Let us remark that this notion of solution is reminiscent of the definition of viscosity solution introduced by Crandall and Lions \citep{crandall1983viscosity}. We feel that it is useful to remark that, for the master equation (\ref{sme}), we could have defined a weak solution by the fact that the function $\phi_{V,y}$ was a viscosity super-solution of a certain PDE for all $V \in \mathbb{R}^d, y \in \mathbb{O}_d$. However, such a formalism would not have allow us to define solution of the master equation for the cases of optimal stopping or impulse control, that is why we prefer the definition of solutions we just presented.
\end{Rem}

\begin{Rem}
Let us remark that this notion of solutions could be stated in more general domains than $\mathbb{O}_d$, as long as we have a boundary condition of the type of hypothesis \ref{stab}.
\end{Rem}
Let us insist that in the previous definition $x_0$ may be on the boundary of $\mathbb{O}_d$. We introduce the ball $B^1_R$ because we shall place ourselves under hypothesis \ref{grow}, which, as we already mentioned, has the effect to bound the trajectories. Before commenting our definition of solution, let us state the two following results which justify this choice of definition.
\begin{Prop}
Under hypotheses \ref{stab} and \ref{grow}, a smooth solution of (\ref{sme}) is also a monotone solution of (\ref{sme}) in the sense of definition \ref{defmons}. %Moreover, under the assumptions of proposition \ref{existstat}, there exists such a monotone solution of (\ref{sme}).
\end{Prop}
\begin{proof}
This result is fairly simple so we only sketch its proof here. Let us consider a classical solution $U$ of (\ref{sme}). If the point of strict minimum of $x \to \langle U(x) - V, x - y \rangle$ is in the interior of $B^1_R$, then thanks to (\ref{eq1}), there is equality in (\ref{smon}). If this point is on the boundary, thanks to the assumptions \ref{stab} and \ref{grow}, the inequality holds.

\end{proof}%Concerning the existence of such a solution, let us remark that if we denote by $U_{\epsilon}$ the solution of the stationary analogue of (\ref{regme}), then $(U_{\epsilon})_{\epsilon > 0}$ is a uniformly lipschitz family of functions and thus, extracting a subsequence if necessary, it converges uniformly toward a limit $U$. Let us consider $V,y,x_0$ such that $x_0$ is a strict minimum of $x \to \langle U_{\epsilon}(x)- V, x - y \rangle$. The following holds.
%\begin{equation}
%\begin{aligned}
%r\langle U_{\epsilon}(x_0), x_0 - y\rangle &+ \lambda \langle U_{\epsilon}(x_0) - T^*U_{\epsilon}(Tx_0),x_0 - y\rangle \geq\\
%& \langle G(x_0, U_{\epsilon}(x_0)), x_0 - y\rangle + \langle F(x_0,U_{\epsilon}(x_0)), U_{\epsilon}(x_0) - V\rangle+ \\
%& +\epsilon \sum_{j = 1}^d\left( -2\sigma_j(x_j)\partial_j U_{\epsilon}^j + \sigma'(x_j)(x_j - y_j)\partial_j U_{\epsilon}^j\right).
%\end{aligned}
%\end{equation}
%We then conclude as in the proof of proposition \ref{stabs} that $U$ is a monotone solution of (\ref{sme}).
%\end{proof}
\begin{Theorem}\label{uniqstat}
Under hypotheses \ref{stab}-\ref{mon}, there exists at most one continuous monotone solution of (\ref{sme}) in the sense of definition \ref{defmons}. If it exists it is a monotone application.
\end{Theorem}
\begin{proof}
Let us consider $U$ and $V$ two such solutions. Let us define $W: \mathbb{O}_d^2 \to \mathbb{R}$ by 
\begin{equation}
W(x,y) = \langle U(x) - U(y), x - y \rangle.
\end{equation}
Thanks to lemma \ref{stegall}, for any $\epsilon > 0$, there exists $(a,b)\in \mathbb{R}^{2d}, |a| \leq \epsilon$, such that $(x,y) \to \langle U(x) - V(y) , x - y \rangle + \langle a,x\rangle + \langle b, y\rangle $ has a strict minimum on $(B_R^1)^2$ (for $R> 0$ chosen sufficiently large chosen independently of $\epsilon$), attained at  $(x_0,y_0)$. Thus because $U$ is a monotone solution the following holds.
\begin{equation}
\begin{aligned}
r\langle U(x_0), x_0 - y_0\rangle &+ \lambda\langle U(x_0) - T^*U(Tx_0), x_0 - y_0\rangle \geq\\
 &\langle G(x_0, U(x_0)), x_0 - y_0\rangle + \langle F(x_0,U(x_0)), U(x_0) - V(y_0)  + a\rangle.
\end{aligned}
\end{equation}
On the other hand, because $V$ is a monotone solution, we deduce that
\begin{equation}
\begin{aligned}
r\langle V(y_0), y_0 - x_0\rangle &+ \lambda \langle V(y_0) -T^*V(Ty_0), y_0 - x_0\rangle \geq\\
& \langle G(y_0, V(y_0)), y_0 - x_0\rangle + \langle F(y_0,V(y_0)), V(y_0) - U(x_0)  + b\rangle.
\end{aligned}
\end{equation}
Summing the two previous equations, we obtain
\begin{equation}
\begin{aligned}
rW(x_0,y_0) &+ \lambda (W(x_0,y_0) - W(Tx_0,Ty_0))\geq\\
& \langle G(x_0,U(x_0)) - G(y_0,V(y_0)),x_0 - y_0 \rangle+ \langle F(x_0,U(x_0)),a\rangle + \langle F(y_0,V(y_0)),b\rangle+\\
&+ \langle F(x_0,U(x_0)) - F(y_0,V(y_0)),U(x_0) - V(y_0) \rangle.%\\
%&\geq  \langle F(x_0,U(x_0)),a\rangle + \langle F(y_0,V(y_0)),b\rangle.
\end{aligned}
\end{equation}
From this we deduce the following.
\begin{equation}
\begin{aligned}
rW(x_0,y_0) \geq & \lambda ( \langle a, x_0 - Tx_0\rangle + \langle b, y_0 - Ty_0\rangle ) + \langle F(x_0,U(x_0)),a\rangle + \langle F(y_0,V(y_0)),b\rangle.%\\
%&\geq  \langle F(x_0,U(x_0)),a\rangle + \langle F(y_0,V(y_0)),b\rangle.
\end{aligned}
\end{equation}
Because $U$ and $V$ are continuous, we deduce from the fact that $\epsilon$ can be chosen arbitrary small, that for any $R> 0$, for all $\eta > 0$, $W \geq - \eta$ on $B^1_R$. Thus we conclude as in the proof of proposition \ref{uniq2} that $U = V$ and that $U$ is monotone.
\end{proof}

This previous result is a strong justification for our notion of solutions. By considering the proof of this result and the equivalent result in the smooth regime, one can realize that we have simply use all the ingredients useful to prove uniqueness of solutions of the master equation in the monotone regime and use it has a definition of solutions. One can wonder if this notion is not too weak. In the next section we show how it can be sufficient to describe solutions in the optimal stopping case or in the impulse control one. Moreover we now present results of stability and consistency concerning monotone solutions.

\begin{Prop}\label{stabs}
Consider a sequence $(F_n,G_n)_{n \in \mathbb{N}}$ of applications from $\mathbb{O}_d\times \mathbb{R}^d$ into $\mathbb{R}^{2d}$ which converges uniformly over all compact toward $(F,G)$. If for all $n$, $U_n$ is a continuous monotone solution of the master equation (\ref{sme}) associated to $F_n$ and $G_n$ and if $(U_n)_{n \in \mathbb{N}}$ converges uniformly toward $U$, then $U$ is a monotone solution of the master equation associated to $F$ and $G$.
\end{Prop}
\begin{proof}
Let us consider $V \in \mathbb{R}^d, y \in \mathbb{O}_d, R>0$ and $x_0$ a point of strict minimum of $\phi : x \to \langle U(x) - V, x - y\rangle$ in $B^1_R$. From lemma \ref{stegall}, we can consider a sequence $(a_n)_{n \in \mathbb{N}} \in (\mathbb{R}^d)^{\mathbb{N}}$ such that for all $n$, 
\begin{equation}
\begin{cases}
|a_n| \leq n^{-1};\\
\phi_n : x \to \langle U_n(x) - V, x - y\rangle + \langle a_n,x\rangle \text{ has a strict minimum $x_n$ in }B^1_R.
\end{cases}
\end{equation}
Let us now remark that for all $n \geq 0$ :
\begin{equation}
\langle U_n(x_n)- V, x_n - y\rangle + \langle a_n, x_n\rangle \leq \langle U_n(x_0) - V, x_0 - y\rangle + \langle a_n,x_0\rangle.
\end{equation}
From this last inequality, we deduce that $(x_n)_{n \in \mathbb{N}}$ converges toward $x_0$. Finally let us remark that because for all $n \geq 0$, $U_n$ is a monotone solution, we can write
\begin{equation}
\begin{aligned}
r\langle U_n(x_n), x_n - y\rangle &+ \lambda \langle U_n(x_n) - T^*U_n(Tx_n),x_n - y \rangle \geq\\
& \langle G_n(x_n, U_n(x_n)), x_n - y\rangle + \langle F_n(x_n,U_n(x_n)), U_n(x_n) - V  + a_n\rangle.
\end{aligned}
\end{equation}
Passing to the limit in this last expression yields the required result.
\end{proof}
\begin{Rem}
The same type of results can be obtained in the case in which one seeks stability for the terms $\lambda$ and $T$ in (\ref{sme}). This can be done by changing mildly the previous proof.
\end{Rem}
We now show consistency of this notion of solution under an additional monotonicity assumption. That is we show that if a smooth function $U$ is a monotone solution of (\ref{sme}) in the sense of definition \ref{defmons} and that it satisfies an additional monotonicity assumption, then it is a classical solution of (\ref{sme}).
\begin{Prop}\label{consist}
Assume that $U\in W^{2,\infty}$ is a monotone solution of (\ref{sme}) in the sense of definition \ref{defmons}. Assume furthermore that for all $x \in \mathbb{O}_d$, $D_x U(x) > 0$ in the order of positive definite matrix. Then $U$ satisfies 
\begin{equation}
r U^i + (F(x,U)\cdot \nabla_x)U^i + \lambda \left(U^i - (T^*U(t,Tx))^i\right)= G^i(x,U) \text{ in } \{x_i > 0\};
\end{equation}
\begin{equation}
r U^i + (F(x,U)\cdot \nabla_x)U^i + \lambda \left(U^i - (T^*U(t,Tx))^i\right)\leq G^i(x,U) \text{ in } \{x_i =0\}.
\end{equation}
\end{Prop}
\begin{proof}
Let us fix $x_0$ in the interior of $\mathbb{O}_d$. Let us define $\phi_{V,y}$ as in definition \ref{defmons}. Let us remark that
\begin{equation}\label{gradphi}
\nabla_x \phi_{V,y}(x_0) = D_x U(x_0) (x_0 - y) + U(x_0) - V;
\end{equation}
\begin{equation}\label{hessphi}
D^2_x \phi_{V,y}(x_0) = 2 D_x U(x_0) + D^2_xU(x_0)(x_0 - y).
\end{equation}
Let us note that from the assumption on the monotonicity of $U$, there is a neighborhood $\mathcal{O}$ of $x_0$ such that for any point $y \in \mathcal{O}$, the right hand side of (\ref{hessphi}) is strictly positive. Then, taking such a $y$ and choosing $V$ such that (\ref{gradphi}) vanishes, we have that $x_0$ is a point of strict minimum of $\phi_{V,y}$. Because $U$ is a monotone solution of (\ref{sme}), we deduce that
\begin{equation}
r\langle U(x_0), x_0 - y + \langle F(x_0,U(x_0)), D_x U(x_0)(x_0 - y)\rangle \rangle \geq \langle G(x_0, U(x_0)), x_0 - y\rangle.
\end{equation}
This last inequality holds for any $y \in \mathcal{O}$. From this we easily deduce that $U$ satisfies (\ref{sme}) at $x_0$. We argue in the same way when $x_0$ is on the boundary of $\mathbb{O}_d$.
\end{proof}
\begin{Rem}
Let us remark that the assumption $D_x U(x) > 0$ in the order of positive definite matrix on $\mathbb{O}_d$ is usually verified when some strict monotonicity assumption is made on $(G,F)$.
\end{Rem}
\begin{Rem}
As it is usually the case in the MFG theory, the value function does not necessary satisfy the master equation on states where there is no player, but only an inequality. This is reminiscent of the weak solutions studied in \citep{cardaliaguet2015second} for instance.
\end{Rem}

\subsection{The time dependent case}
In this section we present the analogue of definition \ref{defmons} for the case of (\ref{me}). Let us note that in general in the MFG theory, the time regularity is not necessary the main challenge and that we could easily define a notion of monotone solutions which are smooth in time. However we prefer, for completeness, to present this concept for functions which are not necessary smooth in time, even though it makes this section more technical. Following the previous part, we define a monotone solution in the time dependent setting with the following :

\begin{Def}\label{defms}
A function $U : (0,\infty)\times \mathbb{O}_d \to \mathbb{R}^d$ is a monotone solution of (\ref{me}) if 
\begin{itemize}
\item for any $V \in \mathbb{R}^d, y \in \mathbb{O}_d$ and $R> 0$ sufficiently large, for any $(t_0,x)\in (0,\infty) \times B^1_R$ such that $x$ is a point of strict minimum of $x \to \langle U(t_0,x) - V, x - y\rangle$ on $B^1_R$, for any smooth real function $\phi$ such that $\phi(t) < \langle U(t,x), x - y \rangle$ for $t \in (t_0 - \epsilon, t_0)$ for some $\epsilon > 0$ with $\phi(t_0) = \langle U(t_0,x), x - y \rangle$, the following holds :
\begin{equation}\label{testt}
\begin{aligned}
\frac{d}{dt} \phi(t_0) &+ \lambda \langle U(t_0,x_0) - T^*U(t_0,Tx_0),x_0 - y \rangle \geq \\
&\langle F(x,U(t_0,x)), U(t_0,x) - V\rangle + \langle G(x,U(t_0,x)), x - y \rangle.
\end{aligned}
\end{equation} 
\item The initial condition holds.\begin{equation}
U(0,x) = U_0(x) \text{ on } \mathbb{O}_d.
\end{equation}
\end{itemize}
\end{Def}
In some sense, we are treating the time derivative using techniques from viscosity solutions. As in the stationary case, we now present results of existence, uniqueness, stability and consistency for this notion of solutions.
\begin{Prop}
Let $U: (0,\infty) \times \mathbb{O}_d \to \mathbb{R}^d$ be a smooth function, classical solution of (\ref{me}). Then it is a monotone solution of (\ref{me}) in the sense of definition \ref{defms}. Moreover, under the assumptions of proposition \ref{existme}, there exists a monotone solution of (\ref{me}).
\end{Prop}
\begin{proof}
We only state that for any $x,y \in \mathbb{O}_d$ and $t_0 > 0$, for any $\phi$ such that $\phi(t) < \langle U(t,x), x - y \rangle$ for $t \in (t_0 - \epsilon, t_0)$ for some $\epsilon > 0$ with $\phi(t_0) = \langle U(t_0,x), x - y \rangle$, then one has $\frac{d}{dt}\phi(t_0) \geq \langle \partial_t U(t_0,x),x-y\rangle$. The rest of the proof follows easily.
\end{proof}
\begin{Theorem}\label{uniqt}
Under hypotheses \ref{stab} and \ref{mon}, there exists at most one uniformly continuous monotone solution of (\ref{me}) in the sense of definition \ref{defms}.
\end{Theorem}
\begin{proof}
Let us denote by $U$ and $V$ two such solutions. We define $W : [0,\infty) \times \mathbb{O}_d^2 \to \mathbb{R}$ by 
\begin{equation}
W(t,x,y) = \langle U(t,x) - V(t,y), x- y \rangle.
\end{equation}
Our aim is to proceed as usual by proving that $W \geq 0$. Let us assume that there exists $(t_0,x_0,y_0)$ such that $W(t_0,x_0,y_0) < 0$. Let us define the function $Z : [0,t_0]^2 \times \mathbb{O}_d^2 \to \mathbb{R}$ by
\begin{equation}
Z(t,s,x,y) = e^{-\gamma(t+s)} \langle U(t,x) - V(s,y),x-y \rangle + \frac{1}{\alpha}(t-s)^2 + \langle a , x\rangle + \langle b,y\rangle + \delta_1 t + \delta_2 s.
\end{equation}
This function is well defined and depends on the parameters $\gamma, \alpha, a,b,\delta_1$ and $\delta_2$. For any $\gamma$, there exists $\epsilon > 0$ such that if 
\begin{equation}\label{condt}
|a| + |b| + |\delta_1| + |\delta_2| < \epsilon
\end{equation}
then 
\begin{equation}\label{hypwrong}
\min Z < -c 
\end{equation}
 for some $c > 0$ depending only on $\gamma$ and $\epsilon$.  This easily comes from the evaluation of $Z$ at $(t_0,t_0,x_0,y_0)$. From lemma \ref{stegall}, there exists $a,b,\delta_1, \delta_2$ satisfying (\ref{condt}) such that $Z$ has a strict minimum on $[0,t_0]^2 \times \mathbb{O}_d^2$ at $(t_*,s_*,x_*,y_*)$. If $t_*> 0$, then 
 \begin{equation}
 \phi_1 : t \to \langle V(s_*,y_*), x_* - y_*\rangle + e^{\gamma(t-t_*)}\langle U(t_*,x_*) - V(s_*,y_*),x_*-y_* \rangle + e^{\gamma(t + s_*)}\delta_1 (t_* - t)
 \end{equation}
can be chosen as a test function in (\ref{testt}) for $U$. If $s_* > 0$ we can construct the analogue function $\phi_2$ for $V$. Thus if both $t_* > 0 $ and $s_* > 0$, then one obtain 
\begin{equation}\label{int47}
\begin{aligned}
\gamma \langle U(t_*,x_*) &- V(s_*,y_*),x_*-y_* \rangle  + \lambda \langle U(t_*,x_*) - V(s_*,y_*), x_* -y_*\rangle  \geq  \\
&e^{\gamma(t_* + s_*)} \left( \delta_1 + \delta_2  + \langle F(x_*,U(t_*,x_*)),a\rangle + \langle F(y_*,V(s_*,y_*)),b\rangle \right)\\
&+ \langle U(t_*,Tx_*) - V(s_*,Ty_*), T(x_* - y_*)\rangle.
\end{aligned}
\end{equation}
Let us note that a posteriori, choosing $\epsilon$ in (\ref{condt}) as small as we want (what we can do and which does not alter (\ref{int47})), we can contradict (\ref{hypwrong}).

We now treat the case in which $t_* = 0$ (the case $s_* = 0$ is similar). Let us take $\eta, \eta', \eta'' > 0$. We want to show that $Z(t_*,s_*,x_*,y_*) \geq - \eta''$ for a certain choice of the parameters. Because we can choose $\alpha$ as small as we want (without altering the proof of the previous case), we can assume that if $t_* = 0$, then $s_* \leq \eta$. Thus using the continuity of $V$ (if $\eta$ is small enough compared to $\eta'$) that $|V(s_*,y_*) - U_0(y_*)|\leq \eta'$. From which we deduce using the continuity of $U$ (if $\eta'$ is small enough compared to $\eta''$) that $Z(t_*,s_*,x_*,y_*) \geq - \eta''$ which contradicts (\ref{hypwrong}) since $\eta''$ is as small as we want. Thus the function $W$ is positive and we conclude as usual.
\end{proof}

\begin{Prop}\label{propstab}
Consider a sequence $(F_n,G_n)_{n \in \mathbb{N}}$ of applications from $\mathbb{O}_d\times \mathbb{R}^d$ into $\mathbb{R}^{2d}$ which converges uniformly over all compact toward $(F,G)$. If for all $n$, $U_n$ is a continuous monotone solution of the master equation associated to $F_n$ and $G_n$ and if $(U_n)_{n \in \mathbb{N}}$ converges uniformly toward $U$, then $U$ is a monotone solution of the master equation associated to $F$ and $G$.
\end{Prop}
\begin{proof}
Let us take $V \in \mathbb{R}^d, y \in \mathbb{O}_d$ and $R> 0$ sufficiently large, $(t_*,x)\in (0,\infty) \times B^1_R$ such that $x$ is a point of strict minimum of $x \to \langle U(t_*,x) - V, x - y\rangle$ on $B^1_R$, and a smooth real function $\phi$ such that $\phi(t) < \langle U(t,x), x - y \rangle$ for $t \in (t_* - \epsilon, t_*)$ for some $\epsilon > 0$ with $\phi(t_*) = \langle U(t_*,x), x - y \rangle$. Reasoning as in the proof of proposition \ref{stabs}, there exists a sequence $(a_n)_{n \in \mathbb{N}} \in (\mathbb{R}^d)^{\mathbb{N}}$ which converges toward $0$ and such that for all $n \geq 0$, $x \to \langle U_n(t_*,x) - V + a_n,x -y\rangle$ has a strict minimum on $B^1_R$ attained at $x_n$. As in the proof of proposition \ref{stabs}, we obtain that $(x_n)_{n \in \mathbb{N}}$ converges toward $x$. We claim that there exists $(t_n)_{n \in \mathbb{N}} \in (t_* - \epsilon, t_*)^{\mathbb{N}}$ and $(\phi_n)_{n \in \mathbb{N}}$ a sequence of functions in $\mathcal{C}^1(\mathbb{R})$ such that :
\begin{itemize}
\item $ t_n \to t_* \text{ as } n \to \infty.$
\item $\|\phi_n - \phi\|_{C^1} \to 0$ as $n \to \infty$.
\item $\phi_n(t) < \langle U_n(t,x_n), x_n - y \rangle$ for $t \in (t_* - \epsilon, t_n)$.
\item $\phi_n(t_n) = \langle U_n(t_n,x_n), x_n - y \rangle$.
\end{itemize}
Let us detail why such sequences exists. Let us define $u_n(t) = \langle U_n(t,x_n), x_n - y \rangle$. For any $n \geq 0$, extracting a subsequence if necessary, we can assume that 
\begin{equation}
\alpha_n := \inf \left\{ u_n(t) - \phi(t) |t \in \left(t_* - \epsilon, t_* - \frac{1}{n+1}\right)\right\} > 0,
\end{equation}
\begin{equation}
|u_n(t_*) - \phi(t_*)| \leq \frac{\alpha_n}{2}.
\end{equation}
Thus from lemma \ref{stegall}, there exists $|\eta_n| < \min\{\frac{1}{n+1};  \frac{\alpha_n}{2(t_*+1)}\}$ such that $t \to u_n(t) - \phi(t) - \eta_n t$ has a strict minimum on $[t_* - \epsilon, t_*]$. By construction this minimum is in $[t_* - \frac{1}{n+1}, t_*]$. Defining $t_n$ this point of strict minimum and $\phi_n(t) := \phi(t) + \eta_n(t- t_n) + u_n(t_n) - \phi(t_n)$, we obtain the aforementioned sequences.

Because for all $n \geq 0$ $U_n$ is a monotone solution, we obtain the following.
\begin{equation}
\begin{aligned}
\frac{d}{dt} \phi(t_n) + \eta_n &+ \lambda \langle U_n(t_n,x_n) - T^*U_n(t_n,Tx_n),x_n - y \rangle \geq \\
&\langle F_n(x_n,U_n(t_n,x_n)), U_n(t_n,x_n) - V + a_n\rangle + \langle G_n(x,U_n(t_n,x_n)), x_n - y \rangle.
\end{aligned}
\end{equation}
Thus passing to the limit $n \to \infty$ we obtain the required result.
\end{proof}

\subsection{A generalization of this method}
As already mentioned above, the aim of this paper is to present a new notion of solution for MFG master equations and not necessary to enter into too much details on these solutions. However we believe the next remark to be worth mentioning. It has been pointed out to us by Pierre-Louis Lions (Coll\`{e}ge de France).

Let us consider the case $\lambda = 0$. The main argument to establish uniqueness of monotone solutions is to consider two such solutions $U$ and $V$ and to prove that $W$ defined by 
\begin{equation}
W(x,y) = \langle U(x) - V(y), x - y\rangle 
\end{equation}
is positive. Instead, for instance, we could have defined the function $W$ with
\begin{equation}
W(x,y) = \langle U(x) - V(y), \phi(x) - \phi(y)\rangle 
\end{equation}
for some $\phi : \mathbb{O}_d \to \mathbb{R}^d$. Now let us remark that if $D_x\phi(x)$ is an invertible matrix for any $x$ in the interior of $\mathbb{O}_d$, the property 
\begin{equation}
W \geq 0 \Rightarrow U = V
\end{equation}
still holds. This remark immediately generalizes the result of this section to a wider class systems. Indeed by replacing the condition $(G,F)$ is monotone by for any $x,y \in \mathbb{O}_d, U,V \in \mathbb{R}^d$
\begin{equation}
\langle G(x,U) - G(y,V), \phi(x) - \phi(y) \rangle + \langle F(x,U) - F(y,V), U.D_x \phi(x) - V.D_x \phi(y) \rangle \geq 0,
\end{equation}
we obtain the uniqueness of the associated monotone solutions. This concept of solution depending on $\phi$ can be defined (in the stationary case) by 
\begin{Def}\label{defmongen}
A function $U \in \mathcal{C}(\mathbb{O}_d, \mathbb{R}^d)$ is said to be a $\phi$-monotone solution of (\ref{sme}) if for any $V \in \mathbb{R}^d, y \in \phi(\mathbb{O}_d)$, $R>0$ sufficiently large and $x_0$ a point of strict minimum of  $\psi_{V,y} : x \to \langle U(x) - V, \phi(x) - y\rangle$ in $B^1_R$, the following holds
\begin{equation}\label{smon}
\begin{aligned}
r\langle U(x_0), \phi(x_0) - y\rangle  \geq \langle G(x_0, U(x_0)), x_0 - y\rangle + \langle F(x_0,U(x_0)), (U(x_0)- V)D_x\phi(x_0) \rangle.
\end{aligned}
\end{equation}
\end{Def}

\section{The master equation for mean field games with optimal stopping}\label{sectstop}
 This section introduces the formulation of the master equation modeling a MFG in which the players have the possibility to leave the game. As already mentioned in the introduction, let us recall that MFG of optimal stopping have been the subject of several works but that the case of the master equation for such MFG has not been treated up to now. 
 
 In this section we are interested with a MFG, similar to the ones represented by (\ref{me}) and (\ref{sme}), except for the fact that the players are allowed to leave the game whenever they decide by paying a certain exit cost. Moreover, once they have left the game, they do not interact anymore with the other players. We refer to \citep{bertucci2018optimal} for more details on this kind of MFG.
 
 Following \citep{bertucci2018optimal} we study first a penalized version of the game, and then show how we can pass to the limit.
 
 \subsection{The penalized master equation}
 In the penalized version of the MFG of optimal stopping, the players cannot decide to leave instantly the game, they can only control the intensity of a Poisson process which give their exit time, and the intensity of this process is bounded by $\epsilon^{-1}$ for $\epsilon > 0$. Even though we do not want to enter into the precise formulation of this penalized game, let us insist on the fact that those aforementioned Poisson processes are supposed to be independent from one player to the other. The penalized master equation is then of the following form in the time dependent setting.
\begin{equation}\label{stopme}
\begin{aligned}
\partial_t U + \frac{1}{\epsilon}\beta(U) + \left(\frac{1}{\epsilon}\beta'(U)*x + F(x,U)\right)&\cdot \nabla_xU + \\
 + \lambda (U - T^*U(t,Tx))= G(x,U) &\text{ in } (0,\infty) \times \mathbb{O}_d,\\
U(0,x) = U_0(x) &\text{ in } \mathbb{O}_d,
\end{aligned}
\end{equation}
where $\beta$ stands for the positive part. We recall that $*$ stands for the term by term product and that $\beta(U)$ is understood component wise. Clearly $\beta'$ is not well defined but we shall come back on this technicality later. In the stationary setting, the form of the penalized master equation is then 
\begin{equation}\label{stopsme}
r U + \frac{1}{\epsilon}\beta(U) + \left(\frac{1}{\epsilon}\beta'(U)*x + F(x,U)\right)\cdot \nabla_xU + \lambda (U - T^*U(t,Tx))= G(x,U) \text{ in }  \mathbb{O}_d.
\end{equation}
 In the case $\lambda = 0$, the form of the characteristics of the previous equations is given by
 \begin{equation}\label{charstop}
\begin{cases}
\dot{V}(t) = G(y(t),V(t)) - \frac{1}{\epsilon}\beta(V),\\
\dot{y}(t) = F(y(t),V(t)) + \frac{1}{\epsilon}\beta'(V)*y.
\end{cases}
\end{equation}
This type of characteristics is clearly what we expect from the study in \citep{bertucci2018optimal}.

In (\ref{stopme})-(\ref{stopsme}), the exit cost paid by the players to leave the game is $0$. The case in which the exit cost depends on the state of the player which is leaving the game but not on the density of all the players can easily be treated in a similar fashion. The case in which the exit cost depends on the density of players is much more involved as structural assumptions have to be made on the form of the exit cost to ensure the propagation of monotonicity. We refer to \citep{bertucci2018optimal} for more detailed on this topic and we leave this case for future research.

For the rest of this study we focus on the stationary case and we mention the time dependent setting at the end of this part on optimal stoping.

\subsection{Results on the stationary penalized equation}
This section presents three, quite simple, results on the penalized master equation which enable us to pass to the limit $\epsilon \to 0$ in the next section. The first result is concerned with the analogue of proposition \ref{uniq2} for (\ref{stopsme}). The second result is a, uniform in $\epsilon$, a priori estimate for the solution of (\ref{stopsme}). The third result is concerned with monotone solutions of (\ref{stopsme}).
\begin{Prop}
Under hypotheses \ref{stab}-\ref{grow}, there exists at most one smooth solution $U$ of (\ref{stopsme}) such that $\beta'(U)$ is understand as being a function which satisfies
\begin{equation}
\beta'(U^i(x)) \in \partial\beta (U^i(x)),
\end{equation}
where $\partial \beta$ is the subdifferential of $\beta$. If it exists, this solution is monotone.
\end{Prop}
\begin{proof}
The proof is very similar to the one of proposition \ref{uniq2}. Let us take $U$ and $V$ two smooth solutions. We introduce $W : \mathbb{O}_d^2 \to \mathbb{R}$ defined by :
\begin{equation}
W(x,y) = \langle U(x) - V(y), x- y\rangle.
\end{equation}
Let us remark that $W$ satisfies 
\begin{equation}
\begin{aligned}
r W + &\left( \frac{1}{\epsilon} \beta'(U(x))*x + F(x,U(x))\right)\cdot \nabla_x W + \left( \frac{1}{\epsilon} \beta'(V(y))*y + F(y,V(y))\right)\cdot \nabla_y W\\
 = &\langle G(x,U(x)) - G(y,V(y)),x-y\rangle + \langle F(x,U(x)) - F(y,V(y)), x - y\rangle +\\
&+ \frac{1}{\epsilon}\left(\langle \beta'U(x)*x - \beta'(V(y))*y, U(x) - V(y)\rangle - \langle \beta(U(x)) - \beta(V(y)), x - y \rangle\right),\\
= & \langle G(x,U(x)) - G(y,V(y)),x-y\rangle + \langle F(x,U(x)) - F(y,V(y)), x - y\rangle +\\
&+ \frac{1}{\epsilon}\left(\langle x, \beta(V(y)) - \beta'(U(x))V(y)\rangle + \langle y, \beta(U(x)) - \beta'(V(y))U(x)\rangle\right),\\
\geq &  \langle G(x,U(x)) - G(y,V(y)),x-y\rangle + \langle F(x,U(x)) - F(y,V(y)), x - y\rangle.
\end{aligned}
\end{equation}
The rest of the proof follows as the one of proposition \ref{uniq2}.
\end{proof}
\begin{Rem}
Two conclusions can be drawn from this result. The first one is that the terms arising from the modeling of optimal stopping only reinforce the monotonicity of the equation. The second one is that we can indeed use the notation $\beta'(\cdot)$ quite freely as it is justified by this uniqueness property.
\end{Rem}
\begin{Rem}
Let us briefly insist on the fact that the addition of the term $\beta'(U)*x$ in the dynamics does not alter the assumptions \ref{stab} and \ref{grow}.
\end{Rem}
The following result is the main argument why we are able to pass to the limit $\epsilon \to 0$ in (\ref{stopsme}).
\begin{Prop}\label{estimatestop}
Under the assumptions of proposition \ref{existstat}, there exists $C > 0$ independent of $ \epsilon$ such that if $U$ is the classical solution of (\ref{stopsme}), then 
\begin{equation}\label{estint}
\|\nabla_x U\| \leq C.
\end{equation}
\end{Prop}
\begin{proof}
This proof is similar to the one of proposition \ref{existstat}. Let us define $W$ and $Z$ by
\begin{equation}
W(x,\xi) = \langle U(x),\xi\rangle,
\end{equation}
\begin{equation}
Z(x,\xi) = \langle \nabla_x W (x, \xi), \xi\rangle - \delta |\nabla_x W|^2 + \gamma |\xi|^2,
\end{equation}
for some constants $\beta$ and $\gamma$. Let us remark that, arguing as if $\beta$ was a smooth function, $Z$ satisfies
\begin{equation}\label{Zedppenal}
\begin{aligned}
& r Z + \langle \epsilon^{-1}\beta'(U)*x + F(x,\nabla_{\xi} W), \nabla_x Z\rangle + \langle (\epsilon^{-1}\beta''(U)*x + D_p F (x,\nabla_{\xi} W )) \nabla_{\xi}Z , \nabla_{x} W\rangle - \\
 &- \langle D_p G(x, \nabla_{\xi}W)\nabla_{\xi}Z, \xi\rangle +  \lambda(Z- Z(t,Tx, T\xi ) )\\
& = \langle D_xG(x, \nabla_x W) \xi, \xi\rangle - \langle D_pG(x, \nabla_{\xi}W) \nabla_x W, \xi\rangle - \langle D_x F(x, \nabla_{\xi} W) \nabla_x W, \xi\rangle \\
&+ \langle D_p F(x, \nabla_{\xi}W) \nabla_x W, \nabla_x W\rangle + \langle \epsilon^{-1} \beta''(U)*x*\nabla_x W, \nabla_x W\rangle\\
&- 2 \delta \langle D_x G(x, \nabla_x W)\xi, \nabla_x W\rangle + 2 \delta \langle D_x F(x, \nabla_{\xi} W) \nabla_x W, \nabla_x W\rangle + 2 \delta \langle \beta'(U)*\nabla_x W, \nabla_x W\rangle\\
& + \delta \lambda \bigg{(} |\nabla_x W|^2 - 2 \langle \nabla_x W(Tx, T \xi ), T \nabla_x W\rangle +  | \nabla_x W(Tx, T \xi)|^2 \bigg{)} \\
& + r \delta |\nabla_x W|^2 + 2 \gamma(\langle D_p F (x,\nabla_{\xi} W ) \xi , \nabla_{x} W\rangle - \langle D_p G(x, \nabla_{\xi}W)\xi, \xi\rangle)\\
& -r\gamma |\xi|^2 + \lambda \gamma (|\xi|^2 - |T\xi|^2).
\end{aligned}
\end{equation}
Let us remark that, since $\beta$ is an increasing and convex function, all the terms in the right hand side involving $\epsilon$ are positive. Thus we conclude that there exists an a priori estimate independent of $\epsilon$. To remark that this fact immediately extend to the case $\beta(\cdot) = (\cdot)_+$, it suffices to realize that $\beta$ can approximate uniformly by smooth convex and increasing functions.
\end{proof}
As already mentioned in the previous section, existence of solution of a stationary master equation can be established from estimates such as (\ref{estint}) if some local boundness holds. We here give an assumption for which such a property can be proven.
\begin{hyp}\label{existstop}
For any $p\in \mathbb{R}^d$, $F(0,p) = 0$. Moreover there is a unique solution $V \in \mathbb{O}_d$ of
\begin{equation}
rV + \lambda (V - T^*V) = -G(0,-V).
\end{equation}
\end{hyp}
We believe this assumption to be quite mild as it only assumes that i) once the mass of players reaches $0$, it stays at $0$, ii) the MFG with a $0$ mass of player (which is thus an optimal control problem) is well defined and players remaining do not exit it. Let us remark that there can indeed still be player in the MFG, even if the mass of players is $0$. In such a case, the remaining players do not "see" each other.
\begin{Prop}\label{resume}
For $\epsilon > 0$, under the assumptions of proposition \ref{existstat} and hypothesis \ref{existstop}, %and the assumption that 
%\begin{equation}
%\forall y \in \mathbb{O}_d, \exists V \leq 0, F(y,V) = 0,
%\end{equation}
 there exists a unique monotone solution $U_{\epsilon}$ of (\ref{stopsme}). It is a solution of (\ref{stopsme}) almost everywhere. The sequence $(U_{\epsilon})_{\epsilon > 0}$ is uniformly lipschitz continuous and $(U_{\epsilon}(0))_{\epsilon > 0}$ is a bounded sequence.
\end{Prop}
The proof of this result follows exactly the argument of the previous part.

\subsection{The limit master equation}
We show in this section how we can characterize the value function of a MFG of optimal stopping using the notion of monotone solutions. The main idea consists in characterizing the limit of the sequence $(U_{\epsilon})_{\epsilon > 0}$ of solutions of (\ref{stopsme}) when $\epsilon \to 0$. Let us recall that, because $U_{\epsilon}$ is a monotone solution of (\ref{stopsme}), for any $V \in \mathbb{R}^d, y \in \mathbb{O}_d$ and $x$ point of strict local minimum of $\phi_{V,y} : x \to \langle U_{\epsilon}(x) - V, x - y\rangle$, the following holds.

\begin{equation}
\begin{aligned}
r \langle U_{\epsilon} &(x), x - y \rangle + \lambda \langle U_{\epsilon}(x) - T^*U_{\epsilon}(Tx),x - y\rangle\\
 \geq &\langle G(x,U_{\epsilon}(x)), x- y \rangle + \langle F(x,U_{\epsilon}(x)), U_{\epsilon}(x) - V \rangle +\\
 &+ \frac{1}{\epsilon}\left( \langle \beta'(U_{\epsilon})*x, U_{\epsilon}(x) - V\rangle - \langle \beta(U_{\epsilon}(x)), x - y \rangle \right),\\
= & \langle G(x,U_{\epsilon}(x)), x- y \rangle + \langle F(x,U_{\epsilon}(x)), U_{\epsilon}(x) - V \rangle +\\
 &- \frac{1}{\epsilon}\left( \langle \beta'(U_{\epsilon})*x, V\rangle + \langle \beta(U_{\epsilon}(x)), y \rangle \right).\\
\end{aligned}
\end{equation}
From this computation, we deduce that if $V \in \mathbb{R}^d$ is such that $V \leq 0$, then for any $y \in \mathbb{O}_d$ and $x$ point of strict local minimum of $\phi_{V,y} : x \to \langle U_{\epsilon}(x) - V, x - y\rangle$, we obtain that
\begin{equation}\label{remstop}
\begin{aligned}
r &\langle U_{\epsilon} (x), x - y \rangle + \lambda \langle U_{\epsilon}(x) - T^*U_{\epsilon}(Tx),x - y\rangle \geq \\
&\langle G(x,U_{\epsilon}(x)), x- y \rangle + \langle F(x,U_{\epsilon}(x)), U_{\epsilon}(x) - V \rangle.
\end{aligned}
\end{equation}
As we clearly expect that the limit of $(U_{\epsilon})_{\epsilon> 0} $ (if it exists) is negative, this leads us to the following definition.

\begin{Def}\label{defmonstop}
A function $U \in \mathcal{C}(\mathbb{O}_d, \mathbb{R}^d)$ is said to be a monotone solution of the master equation for the MFG of optimal stopping if 
\begin{itemize}
\item $U \leq 0$,
\item for any $V \in \mathbb{R}^d$ such that $V \leq 0$, for any $y \in \mathbb{O}_d$, $R>0$ sufficiently large and $x_0$ a point of strict minimum of  $\phi_{V,y} : x \to \langle U(x) - V, x - y\rangle$ in $B^1_R$, the following holds
\begin{equation}
\begin{aligned}
r &\langle U (x_0), x_0 - y \rangle + \lambda \langle U(x_0) - T^*U(Tx_0),x_0 - y\rangle \geq \\
&\langle G(x_0,U(x_0)), x_0- y \rangle + \langle F(x_0,U(x_0)), U(x_0) - V \rangle.
\end{aligned}
\end{equation}
\end{itemize}
\end{Def}
Let us insist that the only thing which differs from the non-optimal stopping case is that $U$ has to be negative component-wise and that we only have information for $V$ which are also negative component wise. The existence of such a solution is stated in the next result.

\begin{Theorem}
Under the assumption of proposition \ref{existstat} and hypothesis \ref{existstop}, there exists a monotone solution $U$ of the MFG of optimal stopping in the sense of definition \ref{defmonstop}.
\end{Theorem}
\begin{proof}
For $\epsilon > 0$ we denote by $U_{\epsilon}$ the unique monotone solution of (\ref{stopsme}). From proposition \ref{resume}, we know that, extracting a subsequence if necessary, $(U_{\epsilon})_{\epsilon> 0}$ converges uniformly toward a function $U$. By considering maxima of $U_{\epsilon}$ on $B^1_R$ for $R > 0$ sufficiently large and $\epsilon > 0$, we immediately obtain from maximum like result (thanks to hypothesis \ref{grow}) that $U_{\epsilon}$ is bounded from above by $C \epsilon$ on $B^1_R$, where $C > 0 $ is a constant which depends on $R$ but not on $\epsilon$. Hence we deduce that $U \leq 0$. The rest of the proof follows as in the proof of proposition \ref{propstab} thanks to (\ref{remstop}). Hence we do not detail the rest of the proof. Let us only mention that thanks to lemma \ref{stegall}, using the notations of the proof of proposition \ref{propstab}, we can choose $(a_n)_{n \in \mathbb{N}}$ such that $V - a_n \leq 0$ for all $n \geq 0$.
\end{proof}
We now present a result of uniqueness for this type of solution.
\begin{Theorem}
Under the hypotheses \ref{stab}-\ref{mon}, there exists at most one monotone solution of the MFG of optimal stopping in the sense of definition \ref{defmonstop}.
\end{Theorem}
\begin{proof}
The proof of this statement follows exactly the one of theorem \ref{uniqstat} by remarking that thanks to lemma \ref{stegall}, using the notation of the aforementioned proof, $a$ and $b$ can be chosen such that $V - a \leq 0$ and $U - b \leq 0$.
\end{proof}

\subsection{Comments on this notion of solution in the optimal stopping case}
In this section we want to discuss how the knowledge of the value function of the MFG is helpful to understand the behavior of the population of agents.

First, we expect that when no player is leaving the game, $U$ solves a certain PDE, which is the master equation characterizing the MFG without optimal stopping. The set which corresponds to the fact that no player is leaving is the set $\cap_{i =1}^d \{ U^i < 0\}$. The fact that $U$ satisfies this property can be obtained by two arguments. First, from the PDE satisfied by $U_{\epsilon}$ for $\epsilon > 0$ and then by passing to the limit $\epsilon \to 0$. Secondly by remarking that if $U^i(x) < 0$ for all $i$ and $D_xU(x) > 0$ in the order of positive definite matrix, then the proof of proposition \ref{consist} can be easily adapted to show that $U$ indeed satisfies
\begin{equation}
r U^i + (F(x,U)\cdot \nabla_x)U^i + \lambda \left(U^i - (T^*U(t,Tx))^i\right)= G^i(x,U) \text{ in } \{x_i > 0\}\cap \{ U < 0 \},
\end{equation}
\begin{equation}
r U^i + (F(x,U)\cdot \nabla_x)U^i + \lambda \left(U^i - (T^*U(t,Tx))^i\right)\leq G^i(x,U) \text{ in } \{x_i =0\} \cap \{ U < 0\}.
\end{equation}
Thus in the set $\{U < 0\}$, we can infer the evolution of the population of agents as in the case without optimal stopping.

On the other hand, when players are actually leaving the game, we would like to gain information on how they are leaving the game. Although the question of describing the precise evolution of the population is not the central question of this paper, we indicate formally what happens for the density of players. We refer to \citep{bertucci2018optimal} for more details on this question in the deterministic case. Let us consider a distribution of players $x \in \mathbb{O}_d$ such that it is optimal for some players to leave the game. For this to happen, one must have  $ I(x)\ne \emptyset $ where $I$ is defined by 
\begin{equation}\label{defI}
I(x) = \{ i \in \{1,...,d\} | U^i(x) = 0 \}.
\end{equation}
A natural requirement could be to expect that starting from $x$, the density of players should instantly become $\tilde{x}$ defined by
\begin{equation}\label{pure}
\begin{cases}
\tilde{x}^i = x^i \text{ for } i \notin I(x),\\
\tilde{x}^i = 0 \text{ for } i \in I(x).
\end{cases}
\end{equation}
This type of behavior correspond to considering only symmetric Nash equilibria in pure strategies. From \citep{bertucci2018optimal}, we know that we have to consider Nash equilibria in mixed strategies and that players can play a certain leaving rate. Hence, even though some players are leaving the game in the state $i \in I(x)$, this does not necessary mean that $\tilde{x}^i = 0$. However, because some players are leaving and some are staying in this situation, there is a way to determine $\tilde{x}^i$. Following \citep{bertucci2018optimal}, we expect that $\tilde{x}$ is characterized by
\begin{equation}\label{mixed}
\begin{cases}
\tilde{x}^i = x^i \text{ for } i \notin I(x),\\
U(\tilde{x}) = U(x),\\
G^i (\tilde{x}, U(\tilde{x}))\tilde{x}^i = 0 \text{ for } i \in I(x).
\end{cases}
\end{equation} 
The last line of (\ref{mixed}) stands for the fact that either all the players in state $i$ have left the game, or either it is also optimal to stay in the game and thus one must have $G^i (\tilde{x}, U(\tilde{x})) = 0$ (recall that the exit cost is $0$). Finally, let us remark that when there is a strict monotonicity assumption on $G$, for any $x \in \mathbb{O}_d$, there is at most one $\tilde{x}$ solution of (\ref{mixed}).

\subsection{The time dependent case}
As we already mentioned above, we are not going to treat in full details the time dependent case. However we indicate the natural generalizations of the results above for this problem. The definition of solution of the problem is straightforward from definitions \ref{defms} and \ref{defmonstop}. It is given by
\begin{Def}\label{defmstopt}
A function $U : (0,\infty)\times \mathbb{O}_d \to \mathbb{R}^d$ is a monotone solution of the time dependent MFG of optimal stopping if 
\begin{itemize}
\item $U \leq 0$,
\item for any $V \in \mathbb{R}^d, V \leq 0, y \in \mathbb{O}_d$ and $R> 0$ sufficiently large, for any $(t_0,x)\in (0,\infty) \times B^1_R$ such that $x$ is a point of strict minimum of $x \to \langle U(t_0,x) - V, x - y\rangle$ on $B^1_R$, for any smooth real function $\phi$ such that $\phi(t) < \langle U(t,x), x - y \rangle$ for $t \in (t_0 - \epsilon, t_0)$ for some $\epsilon > 0$ with $\phi(t_0) = \langle U(t_0,x), x - y \rangle$, the following holds :
\begin{equation}\label{testt}
\begin{aligned}
\frac{d}{dt} \phi(t_0) &+ \lambda \langle U(t_0,x_0) - T^*U(t_0,Tx_0),x_0 - y \rangle \geq \\
&\langle F(x,U(t_0,x)), U(t_0,x) - V\rangle + \langle G(x,U(t_0,x)), x - y \rangle.
\end{aligned}
\end{equation} 
\item \begin{equation}
U(0,x) = \min(U_0(x),0) \text{ on } \mathbb{O}_d.
\end{equation}
\end{itemize}
\end{Def}
The fact that the initial condition is only satisfied when $U_0$ is smaller than $0$ is classical feature of optimal stopping problem. The uniqueness of such solutions is a direct adaptation of the proof of theorem \ref{uniqt}. It can be summarizes as follows.
\begin{Theorem}
Under hypotheses \ref{stab} and \ref{mon}, for any continuous function $U_0$, there exists at most one continuous function solution of the time dependent MFG of optimal stopping in the sense of definition \ref{defmstopt}.
\end{Theorem}
The question of existence of such a solution is more involved that in the stationary case. We expect that it can also be proven by penalization by considering the sequence $(U_{\epsilon})_{\epsilon > 0}$ of solutions of (\ref{stopme}) and taking the limit $\epsilon \to 0$. The main argument to consider this method is that, as in the stationary case, a, uniform in $\epsilon$, estimates on $\|D_x U_{\epsilon}\|$ can be established. However, unlike in the case of (\ref{me}), because of the terms in $\epsilon^{-1}$ in (\ref{stopme}), this does not automatically translates into uniform estimates on the time regularity of $U_{\epsilon}$.

\section{The master equation for mean field games of impulse control}
This section generalizes the results of the previous section to a case in which the players have the possibility to use impulse controls. Let us briefly insist on the fact that in finite state space, the evolution of the state of a player is necessary discontinuous. Hence it could be thought that all MFG in finite state space are games of impulse controls. This is not the case. Indeed if one were to detail the game modeled by a master equation of the type of (\ref{me}), then the transition rates between the states, which are controlled by the players, would be either bounded or get a dissuasive cost as they become higher. In this section, we consider the possibility for a player to change instantly its state by paying a certain (finite) cost, this is what we call a MFG of impulse control. We refer to \citep{bensoussan1984impulse} for a complete presentation of impulse control problems and to \citep{bertucci2020fokker} for the study of MFG of impulse control without common noise.

As the formulation of impulse control is very general, and thus difficult to work with, we shall focus at some point in our study on a particular instance of MFG of impulse control. We hope that the forthcoming results convince the reader of the generality of this method. As we did in the optimal case, we shall also focus on the stationary case in this section.

\subsection{Description of the model}
We consider a MFG in finite state space which is described by $F,G,\lambda$ and $T$ as in (\ref{sme}). We add the possibility for a player in the state $i$ to instantly jump to the state $j$ by paying a cost $k_{ij}\geq 0$. The so-called jump operator is then defined by
\begin{equation}\label{defM}
(Mp)^i = \text{min} \{ k_{ij} + p^j |j \in \{1,...d\}\}, \text{ for } p \in \mathbb{R}^d, i \in \{1,...,d\}.
\end{equation}
We shall use the notation $M^i p = (Mp)^i$ for $i \in \{1,...,d\}$. Given the value function $U$ of the MFG of impulse control, the function $MU$ plays formally the same role as an exit cost. Indeed when a player is in a state $i$ and the distribution of other players is $x \in \mathbb{O}_d$, such that 
\begin{equation}\label{impnot}
U^i(x) < M^iU(x),
\end{equation} then it is strictly sub-optimal to use an impulse control. On the other hand, when there is equality in (\ref{impnot}), then it is optimal for players in state $i$ to jump to a state $j$ which reaches the minimum of $M^iU(x)$ in (\ref{defM}).

As in the optimal stopping case, we introduce first a penalized version of the master equation and then explain how we can pass to the limit.

\subsection{The penalized master equation and monotone solutions of the limit game}
Following the previous section, we are interested in a penalized version of the impulse control MFG in which the players cannot exactly jump instantly from one state to another but only control an intensity of jump which is bounded by $\epsilon^{-1}$ for $\epsilon > 0$ (in particular this penalized game is "classical" MFG in finite state space, it only has a particular structure). The master equation associated to such a MFG is

\begin{equation}\label{penalimp} 
\begin{aligned}
rU^i &+ \frac{1}{\epsilon} \beta (U^i - M^iU) +\frac{1}{\epsilon}\left( ( \bold{\alpha}\cdot\bold{1})* x - x\cdot \bold{\alpha} \right)\cdot \nabla_x U^i+  (F(x,U))\cdot \nabla_x U^i +\\
& + \lambda (U^i - (T^*)^i U(Tx)) = G^i(x,U) \text{ in } \mathbb{O}_d, i \in \{1,...,d\}.
\end{aligned}
\end{equation}
where $\bold{1} = (1,...,1)$ and $\alpha = (\alpha_{ij})_{1 \leq i,j\leq d}$ is a matrix made of $d^2$ real functions on $\mathbb{O}_d$ which satisfy
\begin{equation}\label{alpha}
\begin{cases}
\alpha_{ij}(x) = 0 \text{ if } U^i(x) < U^j(x) + k_{ij} \text{ on } \mathbb{O}_d,\\
\sum_{j = 1}^d \alpha_{ij}(x) = 1 \text{ if } U^i(x) \geq U^j(x) + k_{ij} \text{ on } \mathbb{O}_d,\\
0 \leq \alpha_{ij} \leq 1, \sum_{j = 1}^d \alpha_{ij} \leq 1.
\end{cases}
\end{equation}
Formally, for $i, j \in \{1,...,d\}$, the function $\alpha_{ij}$ indicates the proportion of players in state $i$ which chooses to jump (using the impulse control) to state $j$. Let us note that in particular that different jumps may be optimal to use in the same state and thus that we are forced to use this family of functions $(\alpha_{ij})_{1\leq i,j \leq d}$. Let us insist on the fact established in \citep{bertucci2018optimal} for the optimal stopping case, that if we restrict to ourselves to situation in which only a single behavior is optimal in every state, then a Nash equilibrium may not exist. From \citep{bertucci2020fokker}, we expect that some monotonicity property holds for (\ref{penalimp})-(\ref{alpha}). This following result gives a precise statement of this idea.
\begin{Prop}
Under hypothese (\ref{stab})-(\ref{mon}), there exists at most one smooth function $U$ such that $(U,\alpha)$ is a solution of (\ref{penalimp})-(\ref{alpha}). Moreover if such a solution exists, $U$ is monotone.
\end{Prop}
\begin{proof}
As usual we take $(U,\alpha)$ and $(V, \gamma)$ two such solutions and we define $W : \mathbb{O}_d^2 \to \mathbb{R}$ by
\begin{equation}
W(x,y) = \langle U(x)- V(y), x-y\rangle.
\end{equation}
Let us remark that $W$ is a solution of 
\begin{equation}
\begin{aligned}
r W +& \left(F(x,U) + \frac{1}{\epsilon} (( \bold{\alpha}\cdot\bold{1})* x - x\cdot \bold{\alpha}) \right)\nabla_x W + \left(F(y,V) + \frac{1}{\epsilon} (( \gamma\cdot\bold{1})* y - y\cdot \gamma) \right)\nabla_y W\\
 = \langle &G(x,U) - G(y,V), x - y \rangle + \langle F(x,U) - F(y,V), U - V \rangle- \lambda (W(x,\xi) - W(Tx,T\xi))+\\
 & + \frac{1}{\epsilon}\langle (V - MV)_+ - (U- MU)_+, x - y \rangle + \\
&+\frac{1}{\epsilon} \langle( \alpha\cdot\bold{1})* x - x\cdot \alpha -( \gamma\cdot\bold{1})* y +y\cdot \gamma, U - V \rangle.
\end{aligned}
\end{equation}
Let us remark that 
\begin{equation}
\begin{aligned}
\langle (V - &MV)_+ - (U- MU)_+, x - y \rangle + \langle ( \alpha\cdot\bold{1})* x - x\cdot \alpha -( \gamma\cdot\bold{1})* y +y\cdot \gamma, U - V \rangle\\
= &\langle x, (V - MV)_+ - (U- MU)_+ + (U- V)* ( \alpha\cdot\bold{1}) - \alpha \cdot (U - V) \rangle + \\
&\langle y, (U - MU)_+ - (V- MV)_+ + (V- U)* ( \gamma\cdot\bold{1}) - \gamma \cdot (V - U) \rangle,\\
\geq &\langle x, (V - MV)_+  - V* ( \alpha\cdot\bold{1}) + \alpha \cdot V  + (\alpha*k)\cdot\bold{1} \rangle + \\
&\langle y, (U - MU)_+  - U* ( \gamma\cdot\bold{1}) + \gamma \cdot U + (\alpha*k)\cdot \bold{1} \rangle ,\\
\geq & 0,
\end{aligned}
\end{equation}
where $k = (k_{ij})_{1\leq i,j\leq d}$. Thus we deduce that $W$ satisfies
\begin{equation}
\begin{aligned}
r W +& \left(F(x,U) + \frac{1}{\epsilon} (( \bold{\alpha}\cdot\bold{1})* x - x\cdot \bold{\alpha}) \right)\nabla_x W + \left(F(y,V) + \frac{1}{\epsilon} (( \gamma\cdot\bold{1})* y - y\cdot \gamma) \right)\nabla_y W\\
 +& \lambda(W(x,\xi)-W(Tx,T\xi)) \geq 0.
\end{aligned}
\end{equation}
Thus we conclude as we did in proposition \ref{uniq2}, first that $W \geq 0$, and then the required results.
\end{proof}
This result clearly suggests that the master equation for the impulse control MFG is well posed. Indeed, as we shall see, as the solution of the problem are monotone, the notion of monotone solution introduced in section \ref{sectionmon} shall be helpful. As in the case of optimal stopping, let us note consider a monotone solution $U$ of the penalized equation (\ref{penalimp}) for $\epsilon > 0$ satisfying (\ref{alpha}) for a family of functions $(\alpha_{ij})_{1\leq i,j\leq d}$. Thus for any $V \in \mathbb{R}^d$, $y \in \mathbb{O}_d$ and $x_0$ a point of strict local minimum of $ x \to \langle U(x) - V, x - y\rangle$, the following holds :
\begin{equation}
\begin{aligned}
r\langle U(x_0), x_0 - y\rangle &+ \lambda \langle U(x_0) - T^*U(Tx_0),x_0-y\rangle\\
&\geq \langle G(x_0,U(x_0)) - \beta(U(x_0) - MU(x_0)), x_0 - y \rangle +\\
&+ \langle F(x_0,U(x_0)) + \frac{1}{\epsilon} (( \bold{\alpha(x_0)}\cdot\bold{1})* x_0 - x_0\cdot \bold{\alpha(x_0)}), U(x_0) - V\rangle.
\end{aligned}
\end{equation}
From the same calculations of the previous proof, because (\ref{alpha}) holds, the previous equation can be rewritten to obtain the following.
\begin{equation}
\begin{aligned}
r\langle U(x_0),& x_0 - y\rangle + \lambda \langle U(x_0) - T^*U(Tx_0),x_0-y\rangle\geq \langle G(x_0,U(x_0)), x_0 - y \rangle\\
 &+ \langle F(x_0,U(x_0)), U(x_0) - V\rangle - \frac{1}{\epsilon}\langle x_0, (( \bold{\alpha(x_0)}\cdot\bold{1})* V + \bold{\alpha(x_0)})\cdot V  - ( \alpha*k)\cdot \bold{1}\rangle.
\end{aligned}
\end{equation}
Hence, if $V$ satisfies $V \leq MV$, then we obtain
\begin{equation}
\begin{aligned}
r\langle U(x_0), x_0 - y\rangle &+ \lambda \langle U(x_0) - T^*U(Tx_0),x_0-y\rangle \geq \\
&\langle G(x_0,U(x_0)), x_0 - y \rangle + \langle F(x_0,U(x_0)), U(x_0) - V\rangle.
\end{aligned}
\end{equation}
This remark leads us to the following definition.
\begin{Def}\label{defmonimp}
A function $U \in \mathcal{C}(\mathbb{O}_d, \mathbb{R}^d)$ is said to be a monotone solution of the master equation for the MFG of impulse control if 
\begin{itemize}
\item $U \leq MU$,
\item for any $V \in \mathbb{R}^d$ such that $V \leq MV$, for any $y \in \mathbb{O}_d$, $R>0$ sufficiently large and $x_0$ a point of strict minimum of  $\phi_{V,y} : x \to \langle U(x) - V, x - y\rangle$ in $B^1_R$, the following holds
\begin{equation}
\begin{aligned}
r\langle U(x_0), x_0 - y\rangle &+ \lambda \langle U(x_0) - T^*U(Tx_0),x_0-y\rangle \geq \\
&\langle G(x_0,U(x_0)), x_0 - y \rangle + \langle F(x_0,U(x_0)), U(x_0) - V\rangle.
\end{aligned}
\end{equation}
\end{itemize}
\end{Def}
\begin{Rem}
As we can see, this definition is similar to the one of the optimal stopping case, as they only differ in the fact that the constraint of being negative component wise has been replaced by satisfying an inequality associated with the jump operator $M$.
\end{Rem}
\subsection{Results for particular mean field games of impulse control}
In the present section, we 	assume a particular form for the jump operator $M$ and we prove a result of uniqueness for monotone solutions in the impulse control case.

Let us assume that not all impulse jumps are feasible, that is we assume that $k_{ij} = + \infty$ for some $i,j \in \{1,...,d\}$. More precisely we assume the following.
\begin{hyp}\label{clique}
For all $i \in \{1,...,d\}$, there exists no non constant sequence $i = i_0, i_1,...i_{n+1} = i$ such that for all $k \in \{0,...,n\}$,
\begin{equation}
k_{i_{k}i_{k+1}} < + \infty.
\end{equation}
Moreover, for all $i,j \in \{1,...,d\}$, $k_{ij} > 0$.
\end{hyp}
We comment on this assumption after the following result.
\begin{Theorem}
Under the hypotheses \ref{stab}-\ref{mon} and \ref{clique}, there exists at most one solution of the master equation of the MFG of impulse control in the sense of definition \ref{defmonimp}. If it exists, this solution is monotone.
\end{Theorem}
\begin{proof}
From assumption \ref{clique}, there exists an open set $\mathcal{O} \subset \mathbb{R}^d$ such that $0 \in \bar{\mathcal{O}}$ and for any $V \in \mathbb{R}^d, a \in \mathcal{O}$ :
\begin{equation}
V \leq MV \Rightarrow V + a \leq M(V + a).
\end{equation}
The rest of the proof follows exactly the one of theorem \ref{uniqstat} by remarking that thanks to lemma \ref{stegall}, using the notations of the aforementioned proof, $a$ and $b$ can be chosen in $\mathcal{O}$. 
\end{proof}
\begin{Rem}
Hypothesis \ref{clique} restricts the jumps such that it is not possible to do a sequence of successive jumps to exit one state and enter it again. We firmly believe this assumption is purely technical, even though we have not been able to prove it. Let us also mention that no particular assumption of this form is made on $F$, in particular, using "usual" controls, the players can come back to state from which they have jump.
\end{Rem}

\section{A mean field game of entry and exit}
In this section we consider an application of the tools developed in this paper. We consider a market which agents can enter or exit by paying a certain cost, and in which the revenue of the agents in the market is a function of the total number of agents in the market. The simple model we are about to present is closely related to a model for cryptocurrency mining introduced in \citep{bertucci2020mean}, namely the version of the model in the section 4.4.

The number of agents in the market is denoted by the variable $K$. The cost to exit the market is denoted by $s \in \mathbb{R}$ and the cost to enter the market is denoted by $b \in \mathbb{R}$. We assume 
\begin{equation}
b < s.
\end{equation}
Knowing the evolution of the number of agents in the market $(K_t)_{t \geq 0}$, the revenue of an agent in the market is given by
\begin{equation}
\int_0^{\tau} e^{-rt}g(K_t)dt,
\end{equation}
where $g$ is a real function, $r> 0$ is a constant which takes account of the inter-temporal preference rate of the agents and $0\leq \tau \leq + \infty$ is the time at which the agent exits the market. We want to characterize the value function $U$ of this MFG which gives for a number of agents $K$ in the market, the value of the game $U(K)$ for agents in the market. Following section \ref{sectstop}, we introduce first a penalized version of the game in which the player cannot exit or enter the market freely. For $\epsilon > 0$, a plausible penalization of the MFG leads to the following penalized master equation.
\begin{equation}\label{penales}
rU + \frac{1}{\epsilon} \left( \beta'(U-s)K - \beta(b - U)\right)\partial_K U + \frac{1}{\epsilon} \beta(U -s ) = g(K) \text{ in } \mathbb{R}_+,
\end{equation}
where $\beta(\cdot) = (\cdot)_+$. The following result holds true.
\begin{Prop}
Let $U$ be a smooth solution of (\ref{penales}). Assume that $g$ is an increasing and lipschitz function. There exists $C > 0$ depending only on $g$ and $r$ such that
\begin{equation}
\|\partial_K U\|_{\infty} \leq C.
\end{equation}
\end{Prop}
The proof of this statement follows exactly the one of proposition \ref{estimatestop} so we do not detail it here. From this result, we easily deduce that, if $g$ is lipschitz and increasing, there exists a monotone solution of (\ref{penales}). Such a monotone solution $U$ satisfies that for any $V \in \mathbb{R}$, $y \in \mathbb{R}_+$, $K_0$ a point of local strict minimum of $K \to (U(K) - V)(K - y)$
\begin{equation}
\begin{aligned}
rU(K_0)(K_0 - y) \geq &g(K_0)(K_0 - y) + \frac{1}{\epsilon} \left( \beta'(U(K_0) - s)K_0 - \beta(b - U(K_0) )\right)(U(K_0) - V) -\\
&- \frac{1}{\epsilon}\beta(U(K_0) -s)(K_0 - y).
\end{aligned}
\end{equation}
Let us remark that if $b \leq V \leq s$, we obtain :
\begin{equation}
rU(K_0)(K_0 - y) \geq g(K_0)(K_0 - y).
\end{equation}
This remark leads us to the following definition.
\begin{Def}\label{defes}
A continuous real function $U$ is said to be a monotone solution of the MFG of entry and exit if 
\begin{itemize}
\item $b \leq U \leq s$
\item for any $V \in [b,s]$, for any $y \in \mathbb{R}_+$, for any $K_0$ point of strict local minimum of $K \to (U(K) - V)(K - y)$, 
\begin{equation}
rU(K_0)(K_0 - y) \geq g(K_0)(K_0 - y).
\end{equation}
\end{itemize}
\end{Def}
The following result is easily obtained following the previous parts of this article.
\begin{Theorem}
Assume that $g$ is a lipschitz and increasing function. There exists a unique monotone solution of the MFG of entry and exit in the sense of definition \ref{defes}.
\end{Theorem}
\begin{Rem}
Obviously, we only intend to apply the concepts developed in the previous sections here, and extensions of this model could easily be considered, following modeling as in \citep{bertucci2020mean} or simply by adding terms which plays the same role as $F$ and $G$ in the previous sections.
\end{Rem}
\section{Conclusion and future perspectives}
In this paper we have provided a notion of solution to study a class of first order systems of PDE arising from the MFG theory and referred to as master equations. This notion of monotone solutions relies heavily on the so-called monotone structure of the master equation (hypothesis \ref{mon}). Although it was already well known in the MFG community that regularity and uniqueness for the solution of the master equation can be obtained in a monotone regime, we hope that our notion of solution can be helpful to reduce the assumptions made to obtain well-posedness of master equations. Moreover we have shown in this paper that we are able to characterize the value function of MFG involving a variety of novel actions for the players (stopping, jumping, entering) using monotone solutions.

The natural extensions for this notion of solutions are the case in which the master equation takes the form of an infinite dimensional PDE and the case of second order equations. The generalization of this work to infinite dimensional cases is standard and shall be the subject of another work. Concerning the case of second order master equations, defining a notion of monotone solutions for functions which are one time differentiable is straightforward following the techniques we here developed, however to treat functions which are merely continuous is much more involved and shall also be the subject of a future work.

Finally, we do not claim that the notion of monotone solutions is appropriate to address the question of the characterization of a value function in a MFG in a non-monotone regime, in particular because in such a situation, if such a value function exists, it may be discontinuous. Personally, we do not believe this previous problem to be solvable outside under additional structural assumptions on the MFG (potential case, particular couplings, particular information structure...).

\section*{Acknowledgments}
We would like to thank Pierre-Louis Lions (Coll\`ege de France) for pointing out to us the question of optimal stopping in MFG a few years ago and for the numerous discussions we had on this topic.

\bibliographystyle{plainnat}
\bibliography{bibremarks}

\appendix

\section{Two maximum principle results}
\begin{Lemma}\label{max1}
Let $\phi : (0,\infty)\times \mathbb{O}_d \to \mathbb{R}$ be a smooth solution of 
\begin{equation}
\partial_t \phi + F(x,\phi)\cdot\nabla_x \phi - \sum_{i =1}^d \sigma(x^i)\partial_{ii} \phi + \lambda(\phi - \phi(Tx)) \geq 0 \text{ in } (0,\infty)\times \mathbb{O}_d,\\
\phi(0,x) \geq 0 \text{ in } \mathbb{O}_d,
\end{equation}
where $\sigma$ is a positive function with $\sigma (0) = 0$  and $T : \mathbb{O}_d \to \mathbb{O}_d$. Under hypothesis \ref{stab} $\phi$ is a positive function.
\end{Lemma}
\begin{proof}
We reason by contradiction. Arguing as in the proof of lemma 3 in \citep{bertucci2019some}, we can assume without loss of generality that there exists $(t_0,x_0) \in (0,\infty)\times \mathbb{O}_d$ such that
\begin{equation}
\begin{cases}
\partial_t \phi (t_0,x_0) < 0;\\
\phi(t_0,x_0) = 0 \leq \phi(t_0,y), \forall y \in \mathbb{O}_d;\\
 \sum_{i =1}^d \sigma(x_0^i)\partial_{ii} \phi(t_0,x_0) \geq 0;\\
 F(x_0,\phi)\cdot\nabla_x \phi (t_0,x_0) \geq 0.
\end{cases}
\end{equation}
Thus we conclude to a contradiction by evaluating the PDE satisfied by $\phi$ at $(t_0,x_0)$. Hence $\phi$ is positive.
\end{proof}
\begin{Lemma}\label{max2}
Let $\phi : \mathbb{O}_d \to \mathbb{R}$ be a smooth solution of 
\begin{equation}
r \phi + F(x,\phi)\cdot\nabla_x \phi - \sum_{i =1}^d \sigma(x^i)\partial_{ii} \phi + \lambda(\phi - \phi(Tx)) \geq 0 \text{ in } (0,\infty)\times \mathbb{O}_d,\\
\phi(0,x) \geq 0 \text{ in } \mathbb{O}_d,
\end{equation}
where $\sigma$ is a positive function with $\sigma (0) = 0$  and $T : \mathbb{O}_d \to \mathbb{O}_d$. Under hypotheses \ref{stab} and \ref{grow}, $\phi$ is a positive function.
\end{Lemma}
\begin{proof}
Consider $R$ sufficiently large and $x_0 := \text{argmin}\{\phi(x) | x \in B^1_R\}$. Assumptions \ref{stab} and \ref{grow} precisely state that $-F(x,p)$ points inward $B^1_R$ for any $x \in \partial B^1_R, p \in \mathbb{R}^d$. Thus evaluating the PDE satisfied by $\phi$ at $x_0$ we deduce that $\phi(x_0) \geq 0$, hence that $\phi$ is a positive function.
\end{proof}

\section{Finite state MFG on the orthant}
In this section we consider a problem which arises from the discretization of a continuous problem, mainly to highlight the links between a master equation on $\mathbb{O}_d$ and on the simplex. For other examples of master equations in finite state space, we refer to \citep{bayraktar2019finite,bertucci2020mean}. Namely we are interested in the case in which $G$ is given by 
\begin{equation}\label{discrete}
G_i(x,p) = f_i(x) - \sum_{j \ne i}H((p_{j}- p_i)_-) ; (x,p) \in (\mathbb{R}_+)^d \times \mathbb{R}^d;
\end{equation}
where $H : \mathbb{R} \to \mathbb{R}$ is a function with quadratic growth whose second derivative is bounded from above and from below by a non negative constant. We then define $F(x,p) =- D_pG(x,p)x$. The master equation in such a cases is the PDE of unknown $U$ given by
\begin{equation}\label{degmfg}
\partial_t U  + (F(x,U)\cdot \nabla_x)U = G(x, U) \text{ in } (0,T)\times \mathbb{R}^d;
\end{equation}
where for the sake of clarity we omit terms modeling common noise, such terms could be treated following the same technique as in the previous sections. Obviously such a master equation is associated to games in which the number of players is conserved, i.e. 
\begin{equation}
\sum_{i}F_i = 0.
\end{equation}
Let us consider a reduction of (\ref{degmfg}). We define $V:(0,T)\times \mathbb{R}^{d-1} \to \mathbb{R}^{d-1}$ by :
\begin{equation}\label{defV2}
V_i(t,z) = U_i(t,\phi(z)) - U_d(t,\phi(z));
\end{equation}
for $1\leq i \leq d$, where $\phi : \mathbb{R}^{d-1}\to \mathbb{R}$ is defined by $\phi(z) = (z_1,z_2,...,z_{d-1},1 - \sum_i z_i)$. The natural equation to consider for $V$ is
\begin{equation}\label{changmfg}
\partial_t V + (\tilde{F}(z,V)\cdot\nabla_z)V = \tilde{G}(z,V) \text{ in } (0, \infty) \times \Sigma^d;
\end{equation}
where $\Sigma^d:= \{ x \in \mathbb{O}_{d-1} | x_1 + ...+ x_{d-1}= 1\}$ and $\tilde{G}$ and $\tilde{F}$ are defined for $i \in \{1,...,d-1\}$ by
\begin{equation}
\begin{cases}
\tilde{G}^i(z,V) = G^i(\phi(z),\psi(V))- G^d(\phi(z),\psi(V)),\\
\tilde{F}^i(z,V) = F^i(\phi(z),\psi(V)),
\end{cases}
\end{equation}
where $\psi(p) = (p_1,p_2,...,p_{d-1},0)$ for $p \in \mathbb{R}^{d-1}$. The following result then easily follows.
\begin{Prop}
Let us consider two functions $U: (0,\infty) \times \mathbb{R}^d \to \mathbb{R}^d$ and $V: (0,\infty)\times \mathbb{R}^{d-1} \to \mathbb{R}^{d-1}$ which satisfy (\ref{defV2}). Then if $U$ solves (\ref{degmfg}), $V$ solves (\ref{changmfg}).

Moreover, if $x \to (f_i(x))_{1\leq i \leq d}$ is monotone and $H$ is convex, then both $(G,F)$ and $(\tilde{G},\tilde{F})$ are monotone.
\end{Prop}

\end{document}